\documentclass[12pt, reqno]{amsart}
\usepackage{amscd, amssymb, latexsym, amsmath, amscd, enumerate}
\usepackage{amstext, anysize, amsthm}
\usepackage[sc]{mathpazo} 
\usepackage{enumitem}
\usepackage{pst-node}
\marginsize{2.5cm}{2.5cm}{2cm}{2cm}
\usepackage{graphics}
\usepackage{color}
\usepackage[all]{xy}
\usepackage{tikz}
\usepackage{tikz-cd}
\usepackage[square, sort, numbers]{natbib}

\newtheorem{theorem}{Theorem}[section]

\newtheorem{conjecture}[theorem]{Conjecture}
\newtheorem{corollary}[theorem] {Corollary}
\newtheorem{definition}[theorem]{Definition}

\newtheorem{lemma} [theorem]{Lemma}
\newtheorem{notation}[theorem]{Notation}

\newtheorem{proposition}[theorem]{Proposition}
\newtheorem{remark}[theorem]{Remark}
\newtheorem{question}[theorem]{Question}
\numberwithin{equation}{section}

\newcommand{\F}{\mathbb{F}}
\newcommand{\Q}{\mathbb{Q}}
\newcommand{\C}{\mathbb{C}}

\newcommand{\GL}{{\rm GL}}

\newcommand{\Hom}{{\rm Hom}}

\newcommand{\Ind}{{\rm Ind}}

\newcommand{\Norm}{{\rm Norm}}
\newcommand{\Gal}{{\rm Gal}}
\newcommand{\tr}{{\rm tr}}

\newcommand{\Res}{{\rm Res}}

\newcommand{\modulo}{{\rm mod}}

\newcommand{\Ol}{\mathfrak{o}_l}
\newcommand{\OO}{\mathfrak{o}}

\newcommand{\cO}{\mathfrak{o}}

\newcommand{\Span}{{\rm Span}}
\newcommand{\Stab}{{\rm Stab}}

\title{On degenerate Whittaker space for $\GL_4(\cO_2)$}  
\author{Ankita Parashar}
\address{Department of Mathematics, IIT Delhi, Hauz Khas, New Delhi - 110016}
\email{ankitaparashar216@gmail.com}

\author{Shiv Prakash Patel}
\address{Department of Mathematics, IIT Delhi, Hauz Khas, New Delhi - 110016} 
\address{Department of Mathematics, IIT Dharwad, Chikkamalligawad Village, Dharwad - 580007}
\email{shivprakashpatel@gmail.com}

\keywords{Degenerate Whittaker space, Prasad's conjecture, Regular representations}
\subjclass[2020]{Primary: 20G25, Secondary: 20G05, 20C15}

\begin{document}
\begin{abstract}
Let $\cO_2$ be a finite principal ideal local ring of length 2. 
For a representation $\pi$ of $\GL_{4}(\cO_2)$, the degenerate Whittaker space $\pi_{N, \psi}$ is a representation of $\GL_2(\cO_2)$.
We describe $\pi_{N, \psi}$ explicitly for an irreducible strongly cuspidal representation $\pi$ of $\GL_4(\cO_2)$. 
This description verifies a special case of a conjecture of Prasad. We also prove that $\pi_{N, \psi}$ is a multiplicity free representation.
\end{abstract}

\maketitle

\section{Introduction} \label{Sec:1}
Let $F$ be a finite unramified extension of $\Q_p$ or $\F_{p}((t))$. 
Let $\mathfrak{o}$ be the ring of integers of $F$ with a uniformizer $\varpi$. 
For any positive integer $l$, let $\mathfrak{o}_l:=\mathfrak{o}/(\varpi^l)$.
Then $\cO_l$ is a finite principal ideal local ring of length $l$.
Note that $\mathfrak{o}/(\varpi) \cong \mathbb{F}_q$ is a finite field of order $q=p^f$ of characteristic $p>0$. 
If the characteristic of $F$ is zero then $\cO_{l}$ is the the ring of Witt vectors over $\F_{q}$ of length $l$ and, if the characteristic of $F$ is $p$ then $\cO_{l}$ is isomorphic to $\F_{q}[t]/(t^l)$.
We will be working with both the rings simultaneously and our proofs are uniform for both the cases.
We assume that $p \neq 2$.

Let $\GL_{2n}(\mathfrak{o}_l)$ be the group of invertible $2n \times 2n$ matrices over the ring $\cO_{l}$. 
Let 
\begin{align*}
 & P_{n,n} = \left\{ \left( \begin{matrix} g & X \\ 0 & g' \end{matrix} \right) : g, g' \in \GL_{n}(\cO_l), X \in M_{n}(\cO_{l}) \right\} 
  \text{~and~}
  \\
 &
 N = \left\{ \left( \begin{matrix} I & X \\ 0 & I \end{matrix} \right) : X \in M_{n}(\cO_l)  \right\} \cong M_{n}(\cO_l).
\end{align*}
Let $\psi_0 : \mathfrak{o}_l \rightarrow \mathbb{C}^\times$ be a non-trivial additive character which is also non-trivial when restricted to the subgroup $\varpi^{l-1} \mathfrak{o}_l$ for $l > 1$.
Let $\psi : M_{n}(\cO_l) \rightarrow \mathbb{C}^{\times}$ be the additive character defined by $\psi(X) = \psi_0(\tr(X))$ for all $X \in M_{n}(\cO_l)$, where $\tr(X)$ denotes the trace of $X$. 
Since $M_n(\cO_l) \cong N$, we treat $\psi$ as a character of $N$ as well.
Let $(\pi, V)$ be an irreducible representation of $\GL_{2n}(\cO_l)$. 
Define $V_{N,\psi} := \{ v \in V : \pi(X)v=\psi(X)v \,\, \forall X \in N \}$ which is the subspace of $V$ on which $N$ operates by the character $\psi$. 
The space $V_{N, \psi}$ is called the $(N, \psi)$-Whittaker space of $\pi$, which we call degenerate Whittaker space of $\pi$. 
Since $\tr(gXg^{-1}) = \tr(X)$ for all $g \in \GL_{n}(\mathfrak{o}_{l})$ and $X \in M_{n}(\cO_l)$, it follows that $\GL_n(\cO_l)$ sitting inside the subgroup $P_{n,n}$ as $g \mapsto \left( \begin{matrix} g & 0 \\ 0 & g \end{matrix} \right)$ stabilizes  the character $\psi$ of $N$. 
Therefore, $\GL_n(\cO_l)$ operates on the space $V_{N,\psi}$ giving rise to a representation of $\GL_n(\cO_l)$ which we denote by $\pi_{N, \psi}$.
For a representation $(\pi, V)$, we often omit $V$ and write only $\pi$.
In this paper, we will be concerned with the following question.
\begin{question} \label{main question}
Given an irreducible representation $\pi$ of $\GL_{2n}(\mathfrak{o}_{l})$, what is the degenerate Whittaker space $\pi_{N, \psi}$ as a representation of $\GL_n(\mathfrak{o}_l)$?
\end{question}
The above question is a variant of the question considered by Prasad \cite{Prasad2000} for  $l=1$. 
Prasad’s work has inspired  several related works, including ours.
For example, when $l=1$ Balasubramanian-Khurana \cite{Himanshi2022, Himanshi} consider a variant by taking  different characters $\psi$ for $n=2, 3$  and Balasubramanian-Dangodara-Khurana \cite{balasubramanian} consider analogous variant for general $n$. 
Another variant, when the ring $\cO_l$ is replaced by a non-Archimedean local field, this question has been considered by Prasad \cite{Prasad2001} for $\GL_4$ and more recently Pandey-Venketasubramanian \cite{Sanjeev2024} for $Sp_4$.

Now we briefly describe the work of Prasad \cite{Prasad2000} for $l=1$. 
Note that $\mathfrak{o}_1 \cong \mathbb{F}_q$. 
In this case, $(N, \psi)$-degenerate Whittaker space $\pi_{N,\psi}$ is a representation of $\GL_{n}(\mathbb{F}_q)$. 
Recall that for any positive integer $n$, we have $\mathbb{F}_{q^n} ^\times\hookrightarrow \GL_{n}(\mathbb{F}_{q})$ and the embedding is unique up to conjugation.
By the work of Green \cite{Green1955}, it is well known that every irreducible cuspidal representation of $\GL_n(\F_q)$ is associated with a primitive character  of $\F_{q^n}^{\times}$.
Here is the theorem of Prasad \cite[Theorem 1]{Prasad2000}. 
\begin{theorem} \label{DP theorem}
Let $\pi$ be an irreducible cuspidal representation of $\GL_{2n} (\mathbb{F}_q)$ obtained from a primitive character $\theta: \mathbb{F}_{q^{2n}}^{\times} \rightarrow \mathbb{C}^{\times}$. 
Then, as a representation of $\GL_{n}(\mathbb{F}_{q})$,   
$$\pi_{N, \psi} \cong \Ind_{\mathbb{F}_{q^n}^{\times}}^{\GL_{n}(\mathbb{F}_{q})} ( \theta |_{\mathbb{F}_{q^n}^{\times}} ).$$     
\end{theorem}

For $n=1$ and $\pi$ a regular representation of $\GL_2(\cO_2)$, the Question \ref{main question} is answered in a work of the second author with P. Singla \cite{Sp2021, SP2022}. 
In this case, $\pi_{N, \psi}$ is a representation of $\GL_1(\cO_l)$ which is isomorphic to the center of $\GL_2(\cO_l)$.
More precisely, we have the following result \cite[Theorem 1.1]{SP2022}.
\begin{theorem}
Let $\pi$ be an irreducible  regular representation of $\GL_2(\mathfrak{o}_l)$ with central character $\omega_{\pi}$, then $\pi_{N, \psi}$ is one dimensional and $\pi_{N, \psi} \cong \omega_{\pi}$.
\end{theorem}

A more general version of Theorem \ref{DP theorem} has been conjectured by Prasad which we state below.
Let $L$ (respectively, $E$) be an unramified extensions of $F$ of degree $n$ (respectively, $2n$) and $\mathfrak{O}$ (respectively, $\mathcal{O}$) the ring of integers of $L$  (respectively, $E$). 
Write $\mathfrak{O}_l = \mathfrak{O}/(\varpi^l)$ and $\mathcal{O}_l = \mathcal{O}/ (\varpi^l)$ 
then $\mathfrak{O}_{l}^{\times} \hookrightarrow \GL_{n}(\cO_l)$ and $\mathcal{O}_{l}^{\times} \hookrightarrow \GL_{2n}(\cO_l)$.
We also have $\cO_{l} \subset \mathfrak{O}_{l} \subset \mathcal{O}_{l}$.
It is known from Lusztig \cite{Lusztig2004} and Aubert-Onn-Prasad \cite{AM2010} that a strongly cuspidal representation of $\GL_{2n}(\cO_{l})$ is obtained from a strongly primitive character $\theta : \mathcal{O}_{l}^\times \rightarrow \C^{\times}$ (see Section \ref{Sec:2} for more details).   
\begin{conjecture} [Prasad] \label{DP conjecture}
Let $\pi$ be an irreducible strongly cuspidal representation of $\GL_{2n}(\cO_{l})$ associated to a strongly primitive character $\theta : \mathcal{O}_{l}^{\times} \rightarrow \C^{\times}$. Then as a representation of $\GL_{n}(\cO_{l})$, we have 
$$
\pi_{N, \psi} \cong \Ind_{\mathfrak{O}_{l}^{\times}}^{\GL_{n}(\cO_{l})} (\theta|_{\mathfrak{O}_{l}^{\times}}).
$$
\end{conjecture}
We prove the above conjecture of Prasad for $n=2, l=2$, which is one of the main results of this paper.
\begin{theorem} \label{main theorem}
The Conjecture \ref{DP conjecture} is true for $n=2$ and $l=2$.
\end{theorem}
We would like to mention that a few steps, but not all, in the proof of this theorem also work for general $n$ and $l=2$.
Since, we could not achieve all the steps to prove the above conjecture of Prasad for general $n$ (even for $l=2$), we restrict ourselves to the case of $n=2$ and $l=2$.

In addition, we prove the following multiplicity one result.
\begin{theorem} \label{multiplicity one}
 Let $\pi$ be a strongly cuspidal  representation of $\GL_4(\cO_2)$. 
Then  $\pi_{N, \psi}$   is multiplicity free representation of $\GL_2(\cO_2)$, i.e. for any irreducible representation $\sigma$ of $\GL_2(\cO_2)$ we have
 $$ \dim\left( \Hom_{\GL_2(\cO_2)} \left( \pi_{N, \psi}, \sigma\right)  \right)\leq 1.$$
\end{theorem}

The above theorem can be considered as a variant of a theorem of Rallis for $p$-adic groups \cite[Theorem 1]{Prasad2001}.
The theorem of Rallis is indeed a Bessel model case of Gan-Gross-Prasad (GGP) conjecture \cite{GGP2012}.  
More precisely, it is the case of the GGP conjectures when one considers the restriction of representations of the group ${\rm SO}(6)$, which is closely related to $\GL_4$ to the subgroup ${\rm SO}(3)$ which is closely related to $\GL_2$.
The analogous multiplicity one theorem is also true in Prasad's case for $n=2$, which is a consequence of Theorem \ref{DP theorem}. 
Our proof of Theorem \ref{main theorem} uses the above multiplicity one result.

We now give an outline of the paper.
Our proofs depend on the construction of the regular representations of $\GL_{n}(\Ol)$ as given by Hill \cite{Hill1995} and Krakovski-Onn-Singla \cite{Singla2018} which we recall in Section \ref{Sec:2}. 
In particular, an  irreducible strongly cuspidal representation of $\GL_n(\cO_2)$ is induced from certain subgroup, say $T$. 
Our strategy here is to restrict the representation $\pi$ to the $P_{2,2}$ subgroup which we can describe by Mackey theory. 
We write an exhaustive set $\Omega$ of representatives for the double cosets in $T \backslash \GL_4(\cO_2)/P_{2,2}$ in Section \ref{Sec:3}.
We use these representatives for many calculations.
In fact, $\pi|_{P_{2,2}} \cong \underset{\delta}{\bigoplus}\pi^{\delta}$ where $\delta$ varies over a set of representatives of the double cosets mentioned above. 
Then, for every $\delta \in \Omega$ we describe the representation $\pi^{\delta}_{N, \psi}$ as a representation of $\GL_2(\cO_2)$. 
In Section \ref{Sec:3}, we compute the dimension of $\pi_{N, \psi}$ by computing the dimension of every $\pi^{\delta}_{N, \psi}$.
In Section \ref{Sec : 5}, we write down the character of $\pi^{\delta}_{N, \psi}$ for every $\delta$.
In Section \ref{Sec : 6}, we describe all the characters of $I+ \varpi M_{2}(\cO_2)$ appearing in $\pi^{\delta}_{N, \psi}$ when restricted to $I+ \varpi M_2(\cO_2)$.
There are three types of regular representations of $\GL_2(\cO_2)$ as mentioned in Table \ref{Table 1} which helps us to analyse their occurrence in $\pi^{\delta}_{N, \psi}$ and therefore in $\pi_{N, \psi}$ one by one.
In fact, we describe all the irreducible representations of $\GL_2(\cO_2)$ which appear in $\pi_{N, \psi}$ without appealing to the induced representation $\Pi := \Ind_{\mathfrak{O}_{2}^\times}^{\GL_2(\cO_2)}  ( \theta|_{\mathfrak{O}_{2}^\times} )$. 
In this section, we also prove the multiplicity one result mentioned above in Theorem \ref{multiplicity one}.
In Section \ref{Sec : 7}, we describe all the irreducible constituents of $\Pi$  and verify that all these constituents are the same as those which appear in $\pi_{N, \psi}$ with the same multiplicity (in fact, multiplicity one), which in turn proves Theorem \ref{main theorem}.

It should be noted that the proof of Prasad's theorem in \cite{Prasad2000} is indirect in the sense that there is no explicit isomorphism between $\pi_{N, \psi}$ and $\Ind_{\mathbb{F}_{q^n}^{\times}}^{\GL_{n}(\mathbb{F}_{q})} ( \theta |_{\mathbb{F}_{q^n}^{\times}} )$, instead it is proved that the characters of these two representations coincide.
We also prove Theorem \ref{main theorem} indirectly in the similar sense that we do not give any explicit isomorphism between $\pi_{N, \psi}$ and $\Pi$. 
The proof in \cite{Prasad2000} uses character theory, while our proof is a consequence of Mackey theory and Clifford theory since the construction of $\pi$ is totally explicit.

\section{Preliminaries }\label{Sec:2}

Let $F$ be a finite unramified extension of $\Q_p$ or $\F_{p}((t))$, and let $\mathfrak{o}$ be its ring of integers with a uniformizer $\varpi$.
For any positive integer $l$, let $\mathfrak{o}_l=\mathfrak{o}/(\varpi^l)$. 
Note that $\mathfrak{o}/(\varpi) \cong \mathbb{F}_q$ is a finite field of order $q=p^f$ of characteristic $p>0$. 
We assume $p \neq 2$.
For any integer $m$ with $1\leq m<l$, we have the natural quotient map of rings $\mathfrak{o}_{l} \rightarrow \mathfrak{o}_{m}$. 
This quotient map induces a surjective group homomorphism $\GL_{n}(\mathfrak{o}_l)\rightarrow \GL_{n}(\mathfrak{o}_m)$ for any positive integer $n$. 
For any $g \in \GL_{n}(\cO_l), h \in \GL_n(\cO_m)$ we write $\bar{g} \in\GL_{n}(\cO_m)$ for the image of $g$ under the natural map and $\tilde{h} \in \GL_n(\cO_l)$ for any element in the inverse image of $h$ under the natural map.
Define $K^m_l$ to be the kernel of the natural map $\GL_n(\mathfrak{o}_l)\rightarrow \GL_n(\mathfrak{o}_m)$ which is also called the $m$-th principal congruence subgroup of $\GL_n(\mathfrak{o}_l)$. 
Note that $K_{l}^{m}$ are normal subgroups and, for $m \geq l/2$ the subgroups $K_{l}^{m}$ are abelian.
Thus we have the following natural filtration of subgroups
$$
\{ I \} \subset K_{l}^{l-1} \subset \cdots \subset K_{l}^{m-1} \subset K_{l}^{m}\subset \cdots \subset K_{l}^{1} \subset \GL_{n}(\cO_l).
$$
For $1 \leq m <l$, we have $K^m_l/K^{m+1}_l\cong (M_n(\mathbb{F}_q),+)$.
Let $\widehat{M_{n}(\mathfrak{o}_l)}$ denote the Pontryagin dual group of the group $(M_{n}(\mathfrak{o}_l),+)$.
Fix a non-trivial additive character $\psi_0:\mathfrak{o}_{l}\rightarrow \mathbb{C}^\times$ such that $\psi_0|_{\varpi^{l-1}\mathfrak{o}_{l}}\neq 1.$  
For every $A \in M_{n}(\mathfrak{o}_l)$, define $\phi_{A} \in \widehat{M_{n}(\mathfrak{o}_l)}$ by $\phi_{A}(B) = \psi_0( \tr(AB))$, where $\tr$ denote the trace.
The map $A \mapsto \phi_{A}$ defines a group isomorphism between $M_{n}(\mathfrak{o}_l)$ and $\widehat{M_{n}(\mathfrak{o}_l)}$. 
Note that this isomorphism depends on the choice of the character $\psi_0$. 

For any $1\leq i<l$, $ \cO_{l-i}\cong \varpi^i\mathfrak{o}_l\subset \cO_l$ and there is a natural surjection $\cO_l\rightarrow\cO_{l-i}$ written as $x\mapsto\bar{x}$. 
 Given a character $\psi_0:\cO_l\rightarrow\C^\times$, its restriction to the subgroup $\varpi^{i} \cO_l$ can be seen as character of $\cO_{l-i}$ again denoted by $\psi_{0}$ satisfying the following property
 $\psi_0(\varpi ^{i}a)=\psi_0(\bar{a})$ for all $a\in \cO_{l}$.
 
\subsection{Regular representations of $\GL_n(\cO_l)$}

\begin{definition}
\begin{enumerate}[label = {(\alph*)}]
\item  Let $x \in M_{n}(\F_{q})$.
Then $x$ is called regular if its characteristic polynomial is the same as its minimal polynomial. 
Moreover, $x$ is called regular elliptic if its characteristic polynomial is irreducible.
\item For $x \in M_{n}(\F_q)$, the character $\phi_x : M_n(\mathbb{F}_q) \rightarrow \mathbb{C}^{\times}$ is called regular (respectively, regular elliptic) if $x$ is a regular (respectively, regular elliptic) element.
\item An irreducible representation of $\GL_{n}(\mathfrak{o}_l)$ is called regular (respectively, strongly cuspidal) if its restriction to $K^{l-1}_l \cong M_n(\mathbb{F}_q)$ contains a character which is regular (respectively, regular elliptic).
\end{enumerate}
\end{definition}

In the study of representations of $\GL_n(\Ol)$ the regular representations turn out to be basic objects. 
The regular representations of $\GL_n(\cO_l)$ can be constructed using Clifford theory, see G\'erardin \cite{Gerardin}, Hill \cite{Hill1995}, Krakovski-Onn-Singla \cite{Singla2018}, Stasinski \cite{Stasinski}, Stasinski-Stevens \cite{Stevens-Stasinski}. This construction is more intricate when $l$ is odd, see \cite{Singla2018}.
Now, we briefly describe the construction of regular representations of $\GL_n(\OO_l)$ for even $l$. 
\begin{theorem}\label{constrcution}
Consider the group $\GL_n(\cO_l)$ for an even integer $l > 1$ and set
$m = l/2$. 
Let $\rho : K^m_l \rightarrow \C^{\times}$ be a character such that $\rho|_{K^{l-1}_l}$ is a regular character and, let the inertia group of $\rho$ be $I(\rho) : = \left\lbrace g\in \GL_n(\mathfrak{o}_l):\rho(g^{-1}xg)=\rho(x)~\forall ~x\in K^m_l\right\rbrace$.
Then the following hold.
\\
(1) The character $\rho$ extends to  $I(\rho)$ and let $\tilde{\rho}$ be an extension of $\rho$ to $I(\rho)$. Then $\Ind_{I(\rho)}^{\GL_n(\mathfrak{o}_l)}\tilde{\rho}$ is an irreducible regular representation of $\GL_n(\mathfrak{o}_l)$.\\
(2) Suppose $\pi$ is a regular representation of $\GL_n(\mathfrak{o}_l)$ such that $\rho$ appears in $\pi|_{K^m_l}$. Then, there exists a character $\tilde{\rho}$ of $I(\rho)$ which extends $\rho$ such that $\pi=\Ind_{I(\rho)}^{\GL_n(\mathfrak{o}_l)}\tilde{\rho}.$
\end{theorem}

\begin{remark}\label{property of regular representation}
Let $l$ be even and $l=2m$. 
Then, $K^{m}_{l} \cong M_{n}(\cO_{m})$ and any character of $M_{n}(\cO_{m})$ is $\phi_{B}$ for some $B \in M_{n}(\cO_{m})$.
Let $\pi$ be an irreducible regular representation $\GL_{n}(\cO_l)$ such that $\phi_{B}$ appears in  $\pi|_{K^{m}_{l}}$. 
It can be seen by Clifford theory that $\phi_{B'}$ appears in $\pi|_{K^{m}_{l}}$ if and only if $B' = gBg^{-1}$ for some $g \in \GL_n(\cO_m)$.

\end{remark}
\subsection{Representations of $\GL_2(\cO_2)$}
We summarise all the irreducible representations of $\GL_2(\cO_2)$ below depending on their restriction to the first congruence subgroup $K^{1}_{2}$ which is isomorphic to $M_2(\F_q)$. 
Note that a regular matrix in $M_2(\F_q)$ is either non-split semisimple or split non-semisimple or split semisimple. 
It follows that the irreducible regular representation of $\GL_2(\cO_2)$ can be classified into three classes described below.
If an irreducible regular representation $\sigma$ is such that $\sigma|_{K^{1}_{2}}$ contains $\phi_{B}$ with $B$ non-split semisimple (respectively, split non-semisimple, split semisimple) then we say the type of $\sigma$ is non-split semisimple (respectively, split non-semisimple, split semisimple). 
In $M_2(\F_q)$, for a given trace, there is a unique class of scalar matrix, a unique class of split non-semisimple matrices, $\frac{q-1}{2}$ distinct classes of split semisimple matrices and $\frac{q-1}{2}$ distinct classes of non-split semisimple matrices.
\\
Fix a regular character $\phi_B$ of $K^1_2\cong M_2(\F_q)$ and a character $\omega$ of the center $Z$ of $\GL_2(\cO_2)$.
The character $\phi_B$ can be extended to the inertia group $I(\phi_B)$. 
Observe that $Z\cdot K^1_2\subset I(\phi_B)$.
The following table describes the number of irreducible regular representation of different types when its restriction to $Z \cdot K^{1}_{2}$ contains the character $\omega \cdot \phi_{B}$.

\begin{table}[hbt!]
\caption{}\label{Table 1}
\centering
\vspace{0.3 cm}
\begin{tabular}{|c|c|c|c|}
\hline 
S.No. & Type of $\sigma$ &  dimension of $\sigma$  &  No. of $  \sigma $ with $\omega\cdot\phi_B \subset \sigma|_{Z\cdot K^1_2}$ \\
\hline 
\hline 
1. &  non-split semisimple  & $q(q-1)$ & $q+1$ \\
2. & split non-semisimple& $q^2-1$  & $q$ \\
3. & split semisimple &  $q (q+1)$ & $q-1$ \\
\hline 
\end{tabular}
\end{table}
The irreducible representations of $\GL_2(\cO_2)$, which are not regular, are of the type $\sigma \otimes (\chi \circ \det)$, where $\sigma$ is an irreducible representation of $\GL_2(\F_{q})$ considered as a representation of $\GL_2(\cO_2)$ using the natural map $\GL_2(\cO_2) \rightarrow \GL_2(\F_q)$ and $\chi : \cO_{2}^{\times} \rightarrow \C^{\times}$ is a character.

\subsection{Strongly primitive characters}
Let $E$ be an unramified extension of $F$ of degree $n$. 
Let $\mathcal{O}$ be the integral closure of $\mathfrak{o}$ in $E$ with maximal ideal $\mathcal{P} =\varpi\mathcal{O}$. 
Let $\mathcal{O}_l =\mathcal{O}/\mathcal{P}^l$, then $\mathcal{O}_l$ is isomorphic to a free $\mathfrak{o}_l$-module of rank $n$. 
Therefore, $\GL_n(\mathfrak{o}_l)$ can be identified with the group of automorphisms of $\cO_l$-module $\mathcal{O}_l$. 
This identification is determined up to an inner automorphism of $\GL_n(\mathfrak{o}_l)$.
Clearly, for $a\in \mathcal{O}_l$, $x\mapsto ax$ is an $\mathfrak{o}_l$-module endomorphism of $\mathcal{O}_l$ and we get an embedding of the ring $\mathcal{O}_l$ to the ring of endomorphisms of $\mathfrak{o}_l$-module $\mathcal{O}_l$.
Therefore, $\mathcal{O}_l^\times$ becomes a subgroup of $\GL_n(\mathfrak{o}_l)$ under the same embedding.

\begin{definition}
\begin{enumerate}[label = {(\alph*)}]
\item
A character $\phi : \mathcal{O}_1\cong \F_{q^n}\rightarrow \mathbb{C}^\times$ is called primitive if there does not exist a proper subfield $\mathbb{F}_{q^d}$ of $\F_{q^n}$ and a character $\phi_0$ of $\F_{q^d}$ such that  $\phi(x)=\phi_0(\tr_{\F_{q^n}|\F_{q^d}}(x))$ for all $x\in \F_{q^n}$, where $\tr_{\F_{q^n}|\F_{q^d}}$ denote the trace map for the field extension $\F_{q^n}/\F_{q^d}$. 
\item 
For $l >1$, a character $\theta : \mathcal{O}_l^\times \rightarrow \mathbb{C}^\times$ is called strongly primitive if its restriction to the kernel of the map $\mathcal{O}_l^\times\rightarrow \mathcal{O}_{l-1}^\times$, which is isomorphic to $\mathcal{O}_{1}$, is a primitive character.
\end{enumerate}
\end{definition}
Strongly cuspidal representations of $\GL_n(\mathfrak{o}_l)$ are associated to strongly primitive characters of $\mathcal{O}_l^\times$, which we explain for even $l=2m$.
Let $x \in M_{n}(\F_q)$ be a regular elliptic element and $\tilde{x} \in M_{n}(\cO_{m})$ be such that it maps to $x$ under the natural quotient $M_{n}(\cO_m) \rightarrow M_{n}(\F_q)$.
Since $K^{m}_{l} \cong M_{n}(\cO_m)$, $\phi_{\tilde{x}}$ is a character of $K^{m}_{l}$. 
Then the inertia group $I(\phi_{\tilde{x}}) \cong \mathcal{O}_{l}^{\times} K^{m}_{l}$.
Let $\tilde{\phi}_{\tilde{x}}$ be an extension of $\phi_{\tilde{x}}$ to $I(\phi_{\tilde{x}})$. 
Write $\theta = \tilde{\phi}_{\tilde{x}}|_{\mathcal{O}_{l}^{\times}}$.
Then $\theta$ is a strongly primitive character of $\mathcal{O}_{l}^{\times}$ associated to the strongly cuspidal representation $\pi = \Ind_{I(\tilde{\phi}_{\tilde{x}})}^{\GL_{n}(\cO_l)} \tilde{\phi}_{\tilde{x}}$, see Aubert-Onn-Prasad \cite{AM2010}, Lusztig \cite{Lusztig2004}.
More precisely, we have the following result \cite[Theorem~B]{AM2010}.
\begin{theorem}
There is a canonical bijective correspondence between strongly cuspidal representations of $\GL_n(\mathfrak{o}_l)$ and $\Gal(E/F)$ orbits of strongly primitive characters of $\mathcal{O}_l^\times$.
\end{theorem}
\begin{notation}
\begin{enumerate}[label = {(\alph*)}]
\item Let $L$ be a degree 2 unramified extension of $F$ and $\mathfrak{O}$ the integral closure of $\mathfrak{o}$ in $L$.
Let $E$ be a degree 4 unramified extension of $F$ containing $L$ with $\mathcal{O}$ the integral closure of $\mathfrak{o}$ in $E$.
Throughout the paper, we denote $\mathfrak{o}_2,~\mathfrak{O}_2$ and $\mathcal{O}_2$ for $\mathfrak{o}/(\varpi^2), ~\mathfrak{O}/ (\varpi^2)$ and $\mathcal{O}/(\varpi^2)$ respectively.
\item We denote the subgroup  $I+\varpi M_2(\mathfrak{o}_2)$ of $\GL_2(\mathfrak{o}_2)$ by $J^1_2$ and the subgroup $I+\varpi M_4(\mathfrak{o}_2)$ of $\GL_4(\cO_2)$ by $K^1_2$.
\end{enumerate}
\end{notation}

\section{Double cosets $T\backslash G/P$}\label{Sec:4}
For a representation $\pi$ of a finite group $H$ and an automorphism $\tau$ of $H$,  we define a representation $\pi^{\tau}$ given by $\pi^{\tau} := \pi \circ \tau$.
In particular, if $H \subset \tilde{H}$ is a subgroup and $x \in \tilde{H}$ normalises $H$, then we get an automorphism $\tau_{x}$ of $H$ given by $\tau_{x}(y)= x^{-1}yx$. In this case, we write $\pi^{x}$ instead of $\pi^{\tau_{x}}$. 

Throughout this section,  $G=\GL_4(\mathfrak{o}_2)$ and $P=P_{2,2} \subset G$.
We fix an irreducible strongly cuspidal representation $\pi = \Ind_{I(\phi_{x})}^{\GL_{4} (\cO_2)} (\tilde{\phi}_{x})$ where $x$ is a regular elliptic element of $M_{4}(\F_q)$.
We will choose an embedding $\F_{q^4}^{\times} \hookrightarrow \GL_4(\F_q)$ and take $x \in \F_{q^4}^{\times}$.
We will also fix an embedding of $\mathcal{O}_{2}^{\times} \hookrightarrow \GL_{4}(\cO_2)$ such that $\mathcal{O}_{2}^{\times}$ maps onto $\F_{q^4}^{\times}$ under the quotient  map $\GL_4(\cO_2) \rightarrow \GL_4(\F_{q})$.
Since $x$ is a regular elliptic element, we have $I(\phi_{x}) = \mathcal{O}_{2}^{\times}\cdot K^{1}_{2}$ and $I(\phi_{x}) / K^{1}_{2} \cong \F_{q^4}^{\times}$. 
For convenience, we write $T=I(\phi_x)$ and then $\pi = \Ind_T^G \tilde{\phi}_x$.
Therefore, $\pi_{N, \psi} = (\Res^{G}_{P}\pi)_{N, \psi} = (\Res_P^G \Ind_T^G \tilde{\phi}_x)_{N,\psi}$.
By Mackey theory 
\begin{center}
$\Res_{P}^{G} \Ind_{T}^{G} \tilde{\phi}_{x} \cong \underset{\delta \in T \backslash G/P} {\bigoplus} \Ind_{\delta^{-1}T \delta \cap P}^{P} \tilde{\phi}_{x}^{\delta^{-1}}$.
\end{center} 
For $\delta \in T \backslash G/P$, we will write $\pi^{\delta} = \Ind_{\delta^{-1}T \delta \cap P}^{P} \tilde{\phi}_{x}^{\delta^{-1}}$.
Now we describe the double cosets $T \backslash G/P$. 
We write $\bar{P}$ for the image of $P$ in $\GL_4(\mathbb{F}_q)$ and $\bar{T} = T/K^{1}_{2} \cong \F_{q^4}^{\times}$.
Since $K^{1}_{2}$ is the kernel of the quotient map $\GL_4(\cO_2) \rightarrow \GL_4(\F_q)$ and $K^{1}_{2} \subset T$, the quotient map induces a bijection between set of double cosets
$T\backslash \GL_4(\mathfrak{o}_2)/P$ and $\bar{T} \backslash \GL_4(\mathbb{F}_q)/\bar{P}$. 
\begin{lemma} \label{no of distinct double cosets}
The number of distinct cosets $\bar{T}\backslash \GL_4(\F_q)/\bar{P}$ is $q+1$.
\end{lemma}
\begin{proof}
We know that $\GL_4(\mathbb{F}_q)/\bar{P} = 
$ Gr$(4,2)$ the Grassmannian, i.e. the set of all $2$-dimensional subspaces in a $4$-dimensional vector space over $\F_q$. 
For a group $G$ acting on a set $X$ we write, \def\acts{\curvearrowright} $G \acts X.$ Naturally, \def\acts{\curvearrowright}
$\GL_4(\mathbb{F}_q)\acts$ Gr$(4,2)$ transitively and distinct orbits of the restricted action to the subgroup $\bar{T}  \hookrightarrow \GL_4(\mathbb{F}_q)$ gives the distinct double cosets $\bar{T}\backslash \bar{G}/\bar{P}.$

Now, consider $\F_{q^4}$ as a 4-dimensional vector space over $\F_q$. 
A general element of Gr$(4,2)$ can be written as $\Span\{x,y\}$ the span of vectors $x, y \in \F_{q^4}$ which are linearly independent  over $\F_q$. 
Further, a representative of any orbit under the action \def\acts{\curvearrowright} $\F_{q^4}^\times\acts$ Gr$(4,2)$ can be taken to be $\Span\{1,y\}$ where $y\in \mathbb{F}_{q^4}\backslash\mathbb{F}_q$. 
It can be easily verified that $\Span\{1,y\}$ and  $\Span\{1,y'\}$ are in the same orbit if and only if
$y^\prime=\dfrac{py+q}{ry+s}$ for some $p,q,r,s\in \mathbb{F}_q$ where $\left\{ py+q,ry+s \right\}$ is a linearly independent set of vectors of $\mathbb{F}_{q^4}$, which is equivalent to say that $\exists \begin{pmatrix}
    p & q\\
    r & s
\end{pmatrix} \in \GL_2(\F_q)$ such that $y^\prime=\dfrac{py+q}{ry+s}$. 
 Now consider the action \def\acts{\curvearrowright} $\GL_2(\F_q) \acts \F_{q^4}\backslash\F_q$ given by \begin{equation}
\begin{pmatrix}
p & q\\
r & s
\end{pmatrix}\cdot y=\dfrac{py+q}{ry+s},\text{
~~~~where~}\begin{pmatrix}
         p & q\\
         r & s
\end{pmatrix}\in \GL_2(\F_q)\text{~and~} y \in \F_{q^4}\backslash\mathbb{F}_q. \label{Actions}
\end{equation}
Thus, it is clear that there is a bijection between the orbits in the action\def\acts{\curvearrowright} $\F_{q^4}^\times\acts$ Gr(2,4) and the action \def\acts{\curvearrowright} $\GL_2(\F_q) \acts \F_{q^4}\backslash\F_q$.
We use the action \def\acts{\curvearrowright} $\GL_2(\F_q) \acts \F_{q^4}\backslash\F_q$ to compute the number of distinct orbits.
\begin{itemize}
\item If $y\in \F_{q^2}\backslash\F_{q}$, the stabiliser of $\Span\{1,y\}$ is $\F_{q^2}^\times$ and therefore the orbit size of $\Span\{1,y\}$ is $q^2-q$.
\item If $y\in \F_{q^4}\backslash \F_{q^2}$, the stabiliser of $\Span\{1,y\}$ is $\F_{q}^\times$ and therefore the orbit size of $\Span\{1,y\}$ is  $q(q^2-1)$.
\end{itemize}
Thus, there is a single orbit with stabiliser $\F_{q^2}^\times$ and the number of orbits with stabiliser $\F_q^\times$ is $\dfrac{\left|\F_{q^4}\backslash\F_{q^2}\right|}{q(q^2-1)}=q$. 
Hence, the number of distinct double cosets is $q+1$.
\end{proof}

Now we write down an exhaustive, although not exclusive, set of representatives for these double cosets for which we need to fix an embedding of $\F_{q^4}^{\times}$ inside $\GL_4(\F_q)$. 
We fix an embedding of $\F_{q^4}$ in $M_4(\F_q)$ as described below.

Let $\alpha \in \F_{q}^{\times} \backslash \F_{q}^{\times 2}$ and then $\F_{q^2}=\F_{q}(\sqrt{\alpha})$. 
Let $a, b \in \F_{q}$ be such that $a+b\sqrt{\alpha} \in \F_{q^2}^{\times} \backslash \F_{q^2}^{\times 2}$ and 
\begin{align}\label{Norm is square free} \Norm_{\mathbb{F}_{q^2}|\F_q}(a+b\sqrt{\alpha}) = a^2 -b^2 \alpha  \text{~ is a square free element in~} \F_q,
\end{align} 
where $\Norm_{\mathbb{F}_{q^2}|\F_q}$ denotes the norm map for the field extension $\F_{q^2}/\F_{q}$. 
Write $\beta = \sqrt{a+b \sqrt{\alpha}}$. 
Then $\F_{q^4}=\F_{q}(\beta)$.
Note that $\F_{q}(\beta^2)=\F_{q^2}$. For $x \in \F_{q^4}$, the map $\F_{q^4} \rightarrow \F_{q^4}$ given by $y \mapsto xy$ is $\F_{q}$-linear and thus gives rise to an embedding $\F_{q^4} \hookrightarrow M_{4}(\F_{q})$.
Similarly, there is an embedding $\F_{q^2} \hookrightarrow M_{2}(\F_{q})$.
By fixing the ordered basis  $\{1,\beta^2,\beta,\beta^3\}$ of $\F_{q^4}$ as a $\F_{q}$-vector space, we write an explicit embedding which is given by 
\begin{equation}\label{Fq4 embedding}
x= a_0+a_1\beta^2+a_2\beta+a_3\beta^3 \mapsto \begin{pmatrix}
    X_1 & X_2\\
    X_3 & X_1
\end{pmatrix},     
\end{equation} 
where $X_1, X_2, X_3 \in M_2(\F_{q})$ are as follows:
\\
$X_1=\begin{pmatrix}
    a_0	& -a_1(a^2-b^2\alpha)\\
   a_1	& a_0+2aa_1 
\end{pmatrix},~X_2=\begin{pmatrix}
-a_3(a^2-b^2\alpha) & -(a_2+2aa_3)(a^2-b^2\alpha)\\
a_2+2aa_3	& a_3(3a^2+b^2\alpha)+2aa_2
\end{pmatrix}, X_3=\begin{pmatrix}
   a_2	& -a_3(a^2-b^2\alpha)\\
   a_3	& a_2+2aa_3
\end{pmatrix}$. 
Note that $X_1$ and $X_3$ are the images of $a_0 + a_1 \beta^2 $ and $a_2 + a_3 \beta^2$ under the fixed embedding $\F_{q^2}\hookrightarrow M_{2}(\F_{q})$ obtained by the ordered basis $\{1, \beta^2\}$ of $\F_{q^2}$ as $\F_{q}$-vector space. More explicitly,
\begin{equation}\label{F_q^2 embedding}
a_0+a_1\beta^2\mapsto \begin{pmatrix}
   a_0	& -a_1(a^2-b^2\alpha)\\
   a_1	& a_0+2aa_1\end{pmatrix}.
\end{equation}
From now onwards, we identify $\F_{q^4}^{\times}$ as a subgroup of $\GL_4(\F_q)$ and $\F_{q^2}^{\times}$ as a subgroup of $\GL_{2}(\F_{q})$ via the above mentioned embedding.

\begin{definition}
For $u,v,w \in \F_{q}$, define $A_{u,v} =
\begin{pmatrix}
1 & 0 & 0 & 0\\
0 & 1 & 0 & 0\\
0 & u & 1 & 0\\
0 & v & 0 & 1
\end{pmatrix}$ and 
$A_{w}=
\begin{pmatrix}
1 & 0 & 0 & 0\\
0 & 0 & 1 & 0\\
0 & 1 & 0 & 0\\
0 & w & 0 & 1
\end{pmatrix}$. 
\end{definition}

\begin{lemma} \label{double cosets u,v,w}
By taking the embedding $\F_{q^4}^{\times} \hookrightarrow \GL_4(\F_q)$ obtained from Equation (\ref{Fq4 embedding}), an exhaustive, but not exclusive, set of coset representatives of $\bar{T}\backslash \bar{G}/\bar{P}$ is given by
\begin{center}
$\left\lbrace A_{u,v}, A_{w} ~:~u,v,w \in \mathbb{F}_q \right\rbrace$.
\end{center}
\end{lemma}
\begin{proof}
Recall that $\GL_4(\mathbb{F}_q)/\bar{P}\cong $ Gr$(4,2)$ where $\bar{P}$ is the stabiliser of the subspace $W_0=\Span\{ 1,\beta^2 \}\cong \F_{q^2}$, therefore transitivity of the action \def\acts{\curvearrowright}
$\GL_4(\mathbb{F}_q)\acts$ Gr$(4,2)$ 
implies given any $W\in$ Gr$(4,2)$, we get an element $g\in \GL_4(\mathbb{F}_q)$ such that $g\cdot W_0=W.$ 
Note that a representative of an orbit for the action  \def\acts{\curvearrowright} $\F_{q^4}^\times\acts$ Gr$(4,2)$ can be taken to be $W=\Span\{1, y\}$ for $y \in \F_{q^4} \backslash \F_{q}$. 
Write $y = b_0+b_1\beta^2+b_2\beta+b_3\beta^3$  for some $b_0, b_1, b_2, b_3 \in \F_{q}$ then,
\begin{center}
     $W  = \Span\{ 1,b_1\beta^2+b_2\beta+b_3\beta^3 \}$.
\end{center}
\begin{enumerate}
\item[Case 1:] If $b_1 \neq 0$, we can take $b_1=1$ and  get $W=\Span\{1,\beta^2+b_2\beta+b_3\beta^3\}$ 
 where $b_2,b_3\in \mathbb{F}_q$. 
 We define a linear map on $\F_{q^4}$ which maps $W_0$ to $W$ in such a way that $1 \mapsto 1, \beta^2 \mapsto \beta^2+ b_2 \beta +b_3 \beta^3, \beta \mapsto \beta, \beta^3 \mapsto \beta^3$.
 Then a representative for the double coset corresponding to $W$ can be taken to be $A_{b_2, b_3}$.
\item[Case 2:] If $b_1=0, b_2 \neq 0$ we can take $b_2=1$ and get $W=\Span\{1,\beta+b_3\beta^3\}$ where $b_3\in \mathbb{F}_q$. Similar to the Case 1, a representative for the double coset corresponding to $W$ can be taken to be $A_{b_3}$.
\item[Case 3:] If $b_1=b_2=0, b_3\neq 0$ we can take $b_3=1$ and get $W=\Span\{1,\beta^3\}$.
For this subspace, we claim that its representative can be taken to be either $A_{u,v}$  for some $u,v \in \F_q$ or $A_w$ for some $w \in \F_q$.
Note that $\beta^{-3}$ action takes $\Span\{ 1,\beta^{3}\}$ to $\Span\{1,\beta^{-3}\}$.
Write $\beta^{-3}=b_0+b_1\beta^2+b_2\beta+b_3\beta^3$ for some $b_0,b_1,b_2,b_3$.
We claim that either $b_1 \neq 0$ or $b_2 \neq 0$ and then the proof follows from the earlier two cases. 
If $b_1=b_2=0$, then, $\beta^{-3} = b_0+b_3\beta^3$ which gives $1=b_0\beta^3+b_3\beta^6
= b_0\beta^3-(2ab_3)(a^2-b^2\alpha)+(3a^2+b^2\alpha)b_3\beta^2$. 
Comparing the coefficients both sides, we get, $b_0=0$ and $-2ab_3(a^2-b^2\alpha)=1$ and $(3a^2+b^2\alpha)b_3=0$ which implies $3a^2+b^2\alpha=0\implies \Norm_{\F_{q^2}|\F_{q}}(a+b\sqrt{\alpha})=4a^2$, which contradicts our assumption (\ref{Norm is square free}).\qedhere
\end{enumerate}
\end{proof}

\begin{lemma} \label{A_w and A_uv are the same}
For every $w \in \F_{q}$ there exists $u,v \in \F_{q}$ such that $\bar{T} A_{w} \bar{P} = \bar{T} A_{u,v} \bar{P}$.    
\end{lemma}
\begin{proof}
It can be verified that $\bar{T} A_{w} \bar{P} = \bar{T} A_{u,v} \bar{P}$ if and only if 
\begin{align}\label{Double coset u,v and w equation}
& u^2[w(1+2aw)]-uv[1-w^2(3a^2+b^2\alpha)]-v^2[2a+w(3a^2+b^2\alpha)]
\nonumber \\
& +(1+2aw+(a^2-b^2\alpha)w^2)=0.
\end{align}
Therefore, it is enough to prove that for any $w \in \F_{q}$, there always exists  $u,v \in \F_{q}$ such that Equation (\ref{Double coset u,v and w equation}) holds. 
For a given $w \in \F_{q}$, we consider the above equation as a quadratic equation in $u$. 
If the coefficient of $u^2$ is 0, i.e. $w(1+2aw)=0$, then the Equation (\ref{Double coset u,v and w equation}) is linear and one can easily verify that there exists $u, v \in \F_{q}$ satisfying the equation.
Now we assume that the coefficient of $u^2$ is not zero. 
Then Equation (\ref{Double coset u,v and w equation}) has a solution $u \in \F_{q}$ if and only if the discriminant is a square element in $\F_q$. 
The discriminant $D(v)$ is given by 
\begin{center}
    $D(v) = v^2\left((1+2aw)^2 -(a^2-b^2\alpha)w^2\right) -4w(1+2aw)\left((1+aw)^2-(bw)^2\alpha\right)$.
\end{center} 
Clearly, the coefficient of $v^2$ is not zero, otherwise $a^2 - b^2 \alpha \in \F_{q}^{\times 2}$. 
Then, $$\left|\left\lbrace D(v) : v \in \F_{q} \right\rbrace \right| = \frac{q+1}{2}.$$
Since the number of non-square elements in $\F_q$ is $\frac{q-1}{2}$, there exists $v \in \F_{q}$ such that $D(v) \in \F_{q}^{\times 2}$. Hence we are through.
\end{proof}
By using Lemma \ref{double cosets u,v,w} and Lemma \ref{A_w and A_uv are the same} we conclude the following.
\begin{proposition}
The set $\Omega := \{ A_{u,v} : u,v \in \F_{q} \}$ is an exhaustive, but not exclusive, set of representatives for $\bar{T} \backslash \bar{G} / \bar{P}$. 
\end{proposition}
We state the following lemma without proof which provides an equivalent condition for two elements from $\Omega$ representing the same double coset for $\bar{T}\backslash \bar{G}/\bar{P}$.
\begin{lemma}\label{Double cosets}
For $\delta = A_{u,v}, \delta'=A_{u', v'} \in \Omega$, let 
\begin{align*}
C(\delta, \delta') =~ &  u^2-{u^\prime}^2 +uv{u^\prime}^2-u^2u^\prime v^\prime + 2a(v^2{u^\prime}^2-u^2{v^\prime}^2) + 2a(uv-u^\prime v^\prime) \\
& +(3a^2+b^2\alpha)(u^\prime v^\prime v^2-uv{v^\prime}^2)+(v^2-{v^\prime}^2)(a^2-b^2\alpha).
\end{align*}
Then,
$T \delta P = T \delta' P$ if and only if 
 $C(\delta, \delta')=0$.
 In particular,  $T \delta P = TIP$  if and only if $\delta =I$, where $I$ is the $4\times 4$ identity matrix.
\end{lemma}

Here is a list of notations which we will use throughout this paper.
\begin{notation}\label{Notations for matrices}
\begin{enumerate}[label = {(\alph*)}]
\item We denote the identity matrix of a suitable order by $I$, which will be clear from the context.
\item Let $\Omega_0\subset \Omega$ be such that every double coset in $\bar{T}\backslash \bar{G}/\bar{P}$ has a representative in $\Omega_0$ and distinct elements in $\Omega_0$ represent distinct double cosets in $\bar{T}\backslash \bar{G}/\bar{P}$.
\item Since the quotient map $G \rightarrow \bar{G}$ induces a bijection between $T \backslash G /P$ and $\bar{T} \backslash \bar{G} /\bar{P}$, for  a representative $\delta\in \bar{G}$ of a double coset in $\bar{T}\backslash \bar{G}/\bar{P}$ we write its lift to $G$ by the same letter $\delta$ for the corresponding representative of the double coset in $T\backslash G/P$.
\item
For $\delta=A_{u,v} =\begin{pmatrix}
    I & 0\\
    U & I
\end{pmatrix}$ where $U=\begin{pmatrix}
    0 & u\\
    0 & v
\end{pmatrix}$, write 
\\
$L_{\delta}=X_1U-UX_1+X_3-UX_2U$. In particular, $L_{I}=X_3$.
\label{Notation 3.7}
\item If $L_{\delta}$ is invertible, then write
$\mathcal{B}_{\delta}=L_{\delta}(X_1+X_2U)L_{\delta}^{-1}+X_1-UX_2$.
\label{Notation 3.8}
In particular, $\mathcal{B}_{I}=2X_1$ and $\tr(\mathcal{B}_{\delta})= \tr(2X_1)$. \label{Notation L_I AND B_I}
\item Write $X=a_1^2-a_2a_3-2aa_3^2$ and $Y=a_2^2+a_3^2(a^2-b^2\alpha)+2aa_2a_3$.
\end{enumerate}
\end{notation}

Observe that $Y\neq 0$ otherwise $a_2=a_3=0$ and then $x$ will not be a regular elliptic element of $M_4(\F_q)$. 
It can be checked that 
\begin{align}\label{determinant for Lu,v}
\det(L_{\delta})   = -u^2X-uv\left(2aX+Y\right)-v^2\left(2aY+X(a^2-b^2\alpha)\right)+Y.
\end{align}

\section{The dimension of $\pi_{N,\psi}$}\label{Sec:3}
For a finite dimensional representation $\rho$ of a finite group $H$, we write $\Theta_{\rho}$ for the character of $\rho$.
Suppose $\rho'$ be an irreducible representation of the group $H$. 
The multiplicity of $\rho'$ in $\rho$ is defined to be the dimension of $\Hom_{H}(\rho, \rho')$. 
We say that $\rho$ is multiplicity free if the multiplicity of any irreducible representation in $\rho$ is $\leq 1$.
If we define $\left\langle\Theta_{\rho},\Theta_{\rho'}\right\rangle_{H} := \dfrac{1}{|H|}\underset{g\in H}{\sum}\Theta_{\rho}(g) \bar{\Theta}_{\rho'}(g)$, then by \cite[Corollary 2.16]{Fulton-Harris} we have
$$ 
\dim(\Hom_{H}(V,V'))= \left\langle\Theta_{\rho},\Theta_{\rho'}\right\rangle_{H}.
$$
We fix $x \in \F_{q^4}^{\times}$ as in Equation (\ref{Fq4 embedding}) and
consider a strongly cuspidal representation $\pi=\Ind_{T}^{G} \tilde{\phi}_x$ of $\GL_4(\cO_2)$ (see Theorem \ref{constrcution}).
If we write $\pi^{\delta} =\Ind_{\delta^{-1}T \delta \cap P}^{P} \tilde{\phi}_{x}^{\delta^{-1}}$, then by Mackey theory, 
\begin{center}
$\pi_{N,\psi} =\underset{\delta \in \Omega_{0}}{\bigoplus}\pi^{\delta}_{N,\psi}.$ \end{center}
Therefore, $\dim (\pi_{N, \psi}) = \underset{\delta \in \Omega_{0}}{\sum} \dim(\pi^{\delta}_{N, \psi})$.
In what follows we compute the dimension of $\pi^{\delta}_{N, \psi}$ for all $\delta \in \Omega$ with the help of following series of lemmas.
\begin{lemma}\label{dim pi^delta}
Write $\varpi N=\left\lbrace\begin{pmatrix}
    I & \varpi X\\
    0 & I
\end{pmatrix}:X\in M_2(\cO_2)\right\rbrace$.
Then, for every $\delta \in \Omega$, we have 
$$\pi^{\delta}_{N, \psi} \cong \underset{\gamma\in (\delta^{-1} T\delta\cap P) \backslash P/N}{\bigoplus}\Hom_{\varpi N}\left(\tilde{\phi}_{\delta^{-1}x\delta}^{\gamma^{-1}},\psi\right).$$  
\end{lemma}
\begin{proof}
By Mackey theory,
$\pi^{\delta}_{N,\psi} \cong  \underset{\gamma\in (\delta^{-1} T\delta\cap P) \backslash P/N}{\bigoplus}\Hom_{\gamma^{-1}(\delta^{-1} T\delta\cap P)\gamma\cap N}\left(\tilde{\phi}_{\delta^{-1}x\delta}^{\gamma^{-1}},\psi\right).$ 

The Lemma follows from the fact that the subgroup $\gamma^{-1}(\delta^{-1} T\delta\cap P)\gamma\cap N=\varpi N$ for all $\gamma\in (\delta^{-1}T\delta \cap P) \backslash P /N$, which we prove now. 
For any $\gamma \in P$, we have 
$$\gamma^{-1}(\delta^{-1} T\delta\cap P)\gamma\cap N=\gamma^{-1}(\delta^{-1} T\delta\cap P\cap N)\gamma=\gamma^{-1}(\delta^{-1} T\delta\cap N)\gamma.$$ 
Since $\delta^{-1} T\delta~(\modulo~K^{1}_{2})\cong\mathbb{F}_{q^4}^\times$
we get, $(\delta^{-1} T\delta \cap N )~(\modulo~K^{1}_{2}) = \{ I\}$. 
Then, $\delta^{-1}T\delta\cap N= \varpi N$. 
Also, $\varpi N$ is a normal subgroup of $P$, therefore $\gamma^{-1}(\delta^{-1} T\delta\cap N)\gamma=\varpi N$.
\end{proof}
Note that $(\delta^{-1}T\delta \cap P) \backslash P /N$ is in bijection with $(\delta^{-1} \mathbb{F}_{q^4}^\times \delta \cap \bar{P}) \backslash \bar{P}/\bar{N}$.
\begin{lemma}\label{lemma T-P-N}
A set of representatives of the double coset $(\delta^{-1} \mathbb{F}_{q^4}^\times \delta \cap \bar{P}) \backslash \bar{P}/\bar{N}$
 is given as 
\begin{enumerate}[label = {(\alph*)}]
\item For $\delta=I$, $\left\lbrace\begin{pmatrix}
    g_1 & 0\\
    0 & g_2
\end{pmatrix}:g_1\in \F_{q^2}^\times\backslash\GL_2(\F_q),~g_2\in \GL_2(\F_q)\right\rbrace$ is a set of distinct representatives for 
$(\F_{q^4}^\times\cap \bar{P})\backslash \bar{P}/\bar{N}$.
\item For $\delta\neq I$, $\left\lbrace\begin{pmatrix}
    g_1 & 0\\
    0 & g_2
\end{pmatrix}:g_1\in \F_{q}^\times\backslash\GL_2(\F_q),~g_2\in \GL_2(\F_q)\right\rbrace$ is a set of distinct representatives for $(\delta^{-1}\F_{q^4}^\times\delta\cap \bar{P})\backslash \bar{P}/\bar{N}$.
\end{enumerate}
\end{lemma}
\begin{proof}
Clearly, $\bar{P}/\bar{N}\cong \GL_2(\F_q)\times\GL_2(\F_q)$.
Then (a) follows by noting that $\F_{q^4}^\times\cap\bar{P}\cong \F_{q^2}^\times$ and, (b) follows by noting that $\delta^{-1}\F_{q^4}^\times\delta\cap\bar{P}\cong \F_{q}^\times$ for $\delta \neq I$.
\end{proof}

\begin{lemma}\label{Dim trivial}
\begin{enumerate}[label = {(\alph*)}]
    \item $\dim(\pi^{I}_{N,\psi})=q(q-1).$
    \item Let $\delta \neq I$. If $\pi^{\delta}_{N,\psi}\neq 0$ then, $\dim(\pi^{\delta}_{N,\psi})=q(q^2-1)$.
\end{enumerate}
\end{lemma}
\begin{proof}
For $\gamma \in P$ we have 
$$\left\langle \tilde{\phi}_{\delta^{-1}x\delta}^{\gamma^{-1}},\psi\right\rangle_{\varpi N}=1 \text{~if~} \tilde{\phi}_{\delta^{-1}x\delta}^{\gamma^{-1}}|_{\varpi N} =\psi|_{\varpi N}\text{~and~} 0\text{~otherwise.~}$$
Write $\gamma = \begin{pmatrix}
    \tilde{g}_1 & 0\\
    0 & \tilde{g}_2
\end{pmatrix}$ and
using the definition of $\phi_{x}$ on $\varpi N$ we get,
$\tilde{\phi}_{\delta^{-1}x\delta}^{\gamma^{-1}}|_{\varpi N} =\psi|_{\varpi N}$ if and only if the character $X \mapsto \psi_0(\tr((\tilde{g}_2^{-1}\tilde{L}_{\delta}\tilde{g}_1-I)X))$ of $\varpi N$ is trivial (recall $L_{\delta}$ as defined in Notation \ref{Notations for matrices}\ref{Notation L_I AND B_I}), which is equivalent to
$L_{\delta}g_{1}= g_{2}$. 
\begin{enumerate}[label = {(\alph*)}]
\item Let $\delta=I$.
Note that $L_{I}=X_3$. By Lemma \ref{lemma T-P-N} (a), we get 
\begin{align*} 
\dim(\pi^{I}_{N,\psi}) & = \Bigg{|}\Bigg{\{}(g_1,g_2)\in \left(\F_{q^2}^\times\backslash\GL_{2}(\F_q)\right)\times \GL_2(\F_q):g_2=X_3g_1\Bigg{\}}\Bigg{|}
\\
& = q(q-1).
 \end{align*}
\item Let $\delta\neq I$.  
By Lemma \ref{lemma T-P-N} (b), we get
\begin{align}
\dim(\pi^{\delta}_{N,\psi})
& = \dfrac{1}{|\varpi M_2(\mathfrak{o}_2)|}\sum_{\substack{\gamma\in(\delta^{-1} T\delta\cap P) \backslash P/N\\
X\in \varpi M_2(\mathfrak{o}_2)}}\psi_0(\tr((\tilde{g}_2^{-1}\tilde{L}_{\delta}\tilde{g}_1-I)X))
\nonumber \\
& =\Bigg{|}\Bigg{\{}(g_1,g_2)\in \left(\F_{q}^\times\backslash\GL_{2}(\F_q)\right)\times \GL_2(\F_q):g_2=L_{\delta}g_1\Bigg{\}}\Bigg{|}
\nonumber \\
& =
\begin{cases} 
        q(q^2-1) & ~\text{if}~L_{\delta}\text{ is invertible,} \\
        0 & \text{~otherwise.}
\end{cases}
\qedhere
\end{align}
\end{enumerate}
\end{proof}
\begin{lemma}\label{Uniqueness}
There exists exactly one $\delta\in \Omega_0$ such that $\pi^\delta_{N,\psi}=0.$
\end{lemma}
\begin{proof}
First, we prove that there exists a $\delta =A_{u,v} \in \Omega$ for some $u,v \in \F_{q}$ such that $\pi^{\delta}_{N, \psi} =0$. 
From the proof of Lemma \ref{Dim trivial} it follows that $\pi^{\delta}_{N, \psi} =0$ if and only if $L_{\delta}$ is not invertible, i.e. $\det(L_{\delta})=0$.

Recall that $\det(L_{\delta})=
-(u^2X+uv(2aX+Y)+v^2(2aY+X(a^2-b^2\alpha))-Y)$ where $Y \neq 0$, see (\ref{determinant for Lu,v}).
We claim that there exists $u, v \in \F_{q}$ such that $\det(L_{\delta})=0$.

If $X=0$, then $\det(L_{\delta})=0 \Rightarrow (1- uv - 2av^2)Y =0$ and it can be easily verified that there exist $u,v \in F_{q}$ satisfying this condition.
Now we assume $X \neq 0$ and consider the expression for $\det(L_{\delta})$ as a quadratic polynomial in variable $u$, whose discriminant is
$\left(Y^2-4aXY+4X^2b^2\alpha \right)v^2+4XY$.
Note that if for some $v \in \F_{q}$, the discriminant is a square in $\F_{q}$ then there exists a $u \in \F_{q}$ such that $\det(L_{\delta}) =0$.
Write the discriminant as $C_1v^2+C_2$ for $C_1=Y^2-4aXY+4X^2b^2\alpha$ and $C_2=4XY$. 
Note that $C_1 \neq 0$ otherwise,
\begin{center}
 $Y^2-4aXY+4X^2b^2\alpha = (Y-2aX)^2 - 4X^2 (a^2 -b^2 \alpha) =0$ 
\end{center}
i.e. $a^2 - b^2 \alpha$  is a square element in $\F_q$,
which contradicts our assumption in (\ref{Norm is square free}).
Since $C_1 \neq 0$, $| \{ C_1 v^2 + C_2 : v \in F_{q} \}| = \frac{q+1}{2}$ but the number of non-square elements in $\F_q$ is $\frac{q-1}{2}$.
Therefore, there exists a $v \in \F_{q}$ such that $C_1 v^2 + C_2$ is a square in $\F_{q}$.
This completes the proof of the existence of $\delta \in \Omega$ such that $\det(L_{\delta})= 0$.
\\ 

Now we prove the uniqueness, i.e. if $\pi^{\delta}_{N, \psi}=0$ and $\pi^{\delta'}_{N, \psi}=0$ then $T \delta P = T \delta' P$.
Write $\delta =A_{u,v}$ and $\delta' =A_{u',v'}$ then 
\begin{align} \label{det L_delta  1 and 2}
\det(L_{\delta})= 0 \text{~~and ~~}  \det(L_{\delta'})=0.
\end{align}
If $X= 0$, then (\ref{det L_delta  1 and 2}) gives $1-uv-2av^2=0$ and $1-u'v'-2av'^2=0$.
It can be easily verified that $u,v$ and $u', v'$ satisfy the condition in Lemma \ref{Double cosets} giving $T \delta P = T \delta' P$. \\
Now assume that $X \neq 0$.
By eliminating $Y$ from the two equations in (\ref{det L_delta  1 and 2}) 
we get
\begin{align*}
X \cdot C(\delta, \delta') = 0.
\end{align*}
Since $X \neq 0$, using Lemma \ref{Double cosets}, we get $T \delta P = T \delta' P$.
\end{proof}
\begin{proposition} \label{dimension}
Let $\pi$ be a strongly cuspidal representation of $\GL_4(\mathfrak{o}_2)$ then  $$\dim(\pi_{N,\psi})=q^3(q-1).$$
\end{proposition}
\begin{proof}
Note that $\dim(\pi_{N,\psi}) =\underset{\delta\in \Omega_0}{\sum}\dim(\pi^{\delta}_{N,\psi})$.
Recall that $| \Omega_{0}|= q+1$ and by Lemma \ref{Uniqueness}, there exists a unique $\delta \in \Omega_{0}$ such that $\delta \neq I$ and $\pi^{\delta}_{N, \psi} = 0$.
Using Lemma \ref{Dim trivial},  we get
\begin{equation*}
\dim(\pi_{N,\psi}) = q(q-1)+(q-1) \cdot q(q^2-1)
= q^3(q-1). \qedhere 
\end{equation*} 
\end{proof}
One of the main theorems in this paper is that $\pi_{N, \psi}$ is isomorphic to $\Ind_{\mathfrak{O}_{2}^{\times}}^{\GL_2(\mathfrak{o}_2)} ( \theta|_{\mathfrak{O}_{2}^{\times}})$ where $\pi$ is a strongly cuspidal representation of $\GL_4(\cO_2)$ associated to a strongly primitive character $\theta$ of $\mathcal{O}_{2}^{\times}$. 
Proposition \ref{dimension} proves that these two representations have the same dimension.

\section{Characters of $\pi^{\delta}_{N,\psi}$}\label{Sec : 5}
\vspace{0.1 cm}
For a finite dimensional representation $\rho$ of a group we write $\omega_{\rho}$ for its central character, if it exists.
\begin{lemma}
Let $\mathcal{M}, \mathcal{N}$ be subgroups of a finite group $\mathcal{P}$ such that $\mathcal{N}$ is a normal subgroup and $\mathcal{P} = \mathcal{M} \ltimes \mathcal{N}$. 
Let $\psi : \mathcal{N} \rightarrow \C^{\times}$ be a character of $\mathcal{N}$ and $\mathcal{M}_{\psi} = \{ m \in \mathcal{M} : \psi(mnm^{-1}) = \psi(n) \text{~ for all ~} n \in \mathcal{N}\}$.
Let $(\rho, V)$ be a finite dimensional representation of $\mathcal{P}$ and $V_{\mathcal{N}, \psi} = \{ v \in V : \rho(n)v = \psi(n) v \text{~for all~} n \in \mathcal{N} \}$ which is a representation of $\mathcal{M}_{\psi}$ denoted by $(\rho_{\mathcal{N}, \psi}, V_{\mathcal{N}, \psi})$.
The character of $\rho_{\mathcal{N}, \psi}$ is given by
\begin{center}
$\Theta_{\rho_{\mathcal{N}, \psi}} (m) = \dfrac{1}{|\mathcal{N}|} \underset{n \in \mathcal{N}}{\sum} \Theta_{\rho} (mn) \overline{\psi}(n)$.
\end{center}
\end{lemma}
\begin{proof}
The projection map $\Pr : V \rightarrow V_{\mathcal{N}, \psi}$  is given by $\Pr(v) = \frac{1}{|\mathcal{N}|} \underset{n \in \mathcal{N}}{\sum} \rho(n) \overline{\psi} (n) v$. 
Then, 
\begin{align*}\Theta_{\rho_{\mathcal{N}, \psi}} (m) 
& = \tr (\rho_{\mathcal{N}, \psi} (m)) 
= \tr{(\rho(m)|_{V_{\mathcal{N}, \psi}})} 
= \tr{(\rho \circ \Pr)} 
\\
& = \frac{1}{|\mathcal{N}|} \sum_{n \in \mathcal{N}} \tr(\rho(mn)) \overline{\psi}(n) 
= \frac{1}{|\mathcal{N}|} \sum_{n \in \mathcal{N}} \Theta_{\rho}(mn) \overline{\psi}(n).\qedhere
\end{align*}
\end{proof}
The main application of the above lemma is the following.
\begin{corollary} \label{Greens formula}
Let $\rho$ be a finite dimensional representation of $P_{n,n} = M \ltimes N$ where $\mathcal{M} \cong \GL_{n}(\cO_l) \times \GL_{n}(\cO_l)$ and $N \cong M_{n}(\cO_l)$. 
Then, $g \in \GL_{n}(\cO_{l})$ we have
$$
\Theta_{\rho_{N, \psi}}(g) = \dfrac{1}{|N|} \sum\limits_{\substack{{X \in N}}} \Theta_{\rho} \begin{pmatrix} g & X \\ 0 & g \end{pmatrix} \overline{\psi}(Xg^{-1}).
$$
\end{corollary}

We first calculate $\Theta_{\pi^{\delta}_{N, \psi}}$ on $J^{1}_{2} =I + \varpi M_{2}(\cO_2)$, for which we need $\Theta_{\pi^{\delta}}$.
By using the formula for the character of an induced representation,
the value $\Theta_{\pi^{\delta}}\begin{pmatrix}
    I+\varpi A & X\\
    0 & I+\varpi A
\end{pmatrix}$ equals
$$ \dfrac{1}{|\delta^{-1} T\delta\cap P|}\sum\limits_{\substack{S,R\in \GL_2(\mathfrak{o}_2) \\ Q\in M_2(\mathfrak{o}_2) }}\tilde{\phi}_x^{\delta^{-1}}\left(\begin{pmatrix}
     S & Q\\
     0 & R
\end{pmatrix}^{-1}\begin{pmatrix}
    I+\varpi A & X\\
    0 & I+\varpi A
\end{pmatrix}\begin{pmatrix}
    S & Q\\
    0 & R
\end{pmatrix}\right).$$
It is clear that this value is zero if $X\notin \varpi M_2(\cO_2)$. 
Then by using Corollary \ref{Greens formula}, 
$\Theta_{\pi^{\delta}_{N,\psi}}(I+\varpi A)$ is given by
\begin{align}\label{equation 1}
\resizebox{1.0\linewidth}{!}{\begin{tabular}{c}
$\frac{1}{q^{8}|\delta^{-1}T\delta\cap P|}\sum\limits_{\substack{\\ S,R\in \GL_2(\mathfrak{o}_2) \\ X\in \varpi M_2(\mathfrak{o}_2) \\ Q\in M_2(\mathfrak{o}_2)}}  
\tilde{\phi}_x^{\delta^{-1}}
\begin{pmatrix}
    I+\varpi S^{-1}AS & 0\\
    0 & I+\varpi R^{-1}AR
\end{pmatrix}  
\tilde{\phi}_x^{\delta^{-1}}\begin{pmatrix}
    I & \varpi(S^{-1}AQ-S^{-1}QR^{-1}AR)\\
    0 & I
\end{pmatrix}$\vspace{-0.5cm}\\ 
$ \tilde{\phi}_x^{\delta^{-1}}\begin{pmatrix}
    I & S^{-1}XR
    \\
    0 & I
\end{pmatrix} \overline{\psi}(X)$
\end{tabular}}
\end{align}
We compute this by first summing over $X$, then over $Q$, and then over $(S,R)$.

\begin{lemma}\label{S,X,R}
For $\delta \in \Omega$, let $\pi^{\delta}_{N,\psi}\neq 0$ and recall $L_{\delta}$ and $\mathcal{B}_{\delta}$ from Notation \ref{Notations for matrices} \ref{Notation 3.7}, \ref{Notation 3.8}.
\begin{enumerate}[label = {(\alph*)}]
\item For any $S, R \in \GL_2(\cO_2)$ we have
\begin{center}
$\underset{X\in \varpi M_2(\mathfrak{o}_2)}{\sum}\tilde{\phi}_x^{\delta^{-1}}\begin{pmatrix}
    I & S^{-1}XR
    \\
    0 & I
\end{pmatrix}\overline{\psi}(X)=\begin{cases}
    q^4 & \text{if~} S=R\tilde{L}_{\delta}, \\
    0 & \text{otherwise.}
\end{cases}$
\end{center}
\item For any $\delta$ and $S,R\in \GL_2(\cO_2)$ with $S=R \tilde{L}_{\delta}$, 
we have 
\begin{center} $\underset{Q\in M_2(\mathfrak{o}_2)}{\sum}\tilde{\phi}_x^{\delta^{-1}}\begin{pmatrix}
    I & \varpi\left(S^{-1}AQ-S^{-1}QR^{-1}AR\right)\\
    0 & I
\end{pmatrix}=q^8$.
 \end{center}
 \item For $S=R\tilde{L}_{\delta}$, we have
\begin{align}
\tilde{\phi}_x^{\delta^{-1}}\begin{pmatrix}
        I+\varpi S^{-1}AS & 0\\
        0 & I+\varpi R^{-1}AR
    \end{pmatrix}=\psi_0(\varpi~\tr(R\tilde{\mathcal{B}_{\delta}}R^{-1}A)).
\end{align}
\end{enumerate}
\end{lemma}

\begin{proof}
\begin{enumerate}[label = {(\alph*)}]
\item Recall that $\phi_x(I+\varpi A)=\psi_0(\varpi~\tr(\tilde{x}A))$. Since $X\in \varpi M_2(\cO_2)$, 
for any $\delta\in \Omega$, we have
\begin{align*}\label{equation 1 in lemma 5.2}
\underset{X\in \varpi M_2(\mathfrak{o}_2)}{\sum}\tilde{\phi}_x^{\delta^{-1}}\begin{pmatrix}
    I & S^{-1}XR
    \\
    0 & I
\end{pmatrix}\overline{\psi}(X) & =\underset{X\in \varpi M_2(\mathfrak{o}_2)}{\sum}\psi_0\left(\varpi~\tr(    R\tilde{L}_{\delta}S^{-1}-I)X\right)
\\
& =
\begin{cases}
     |\varpi M_2(\mathfrak{o}_2)| & ~\text{if}~ S=R\tilde{L}_{\delta}, \\
    0 & \text{otherwise.}
\end{cases}
\end{align*}
\item Assume $S=R\tilde{L}_{\delta}$. The statement follows if we prove that the character 
$$Q\mapsto\tilde{\phi}_x^{\delta^{-1}}\begin{pmatrix}
    I & \varpi\left(S^{-1}AQ-S^{-1}QR^{-1}AR\right)\\
    0 & I
\end{pmatrix}$$ 
is the trivial character of $M_2(\mathfrak{o}_2)$.
Since $\tilde{\phi}_x^{\delta^{-1}}|_{J^{1}_{2}} = \phi_{\delta^{-1} x \delta}|_{J^1_2}$, we get 
\begin{align*}
    \tilde{\phi}_x^{\delta^{-1}}\begin{pmatrix}
        I & \varpi (S^{-1}AQ-S^{-1}QR^{-1}AR)\\
        0 & I
    \end{pmatrix} &  = \psi_0\left(\varpi~\tr(\tilde{L}_{\delta}(S^{-1}AQ-S^{-1}QR^{-1}AR))\right)
    \\
    & =\psi_0\left(\varpi~\tr(R^{-1}AQ-R^{-1}QR^{-1}AR)\right)
    \\
    & = 1.
\end{align*}

\item Since $S=R\tilde{L}_{\delta}$, we get,
\begin{align*}
\tilde{\phi}_x^{\delta^{-1}}\begin{pmatrix}
        I+\varpi S^{-1}AS  & 0\\
        0 & I+\varpi R^{-1}AR
    \end{pmatrix} & = \tilde{\phi}_x^{\delta^{-1}}\begin{pmatrix}
        I+\varpi \tilde{L}_{\delta}^{-1}R^{-1}AR\tilde{L}_{\delta} & 0\\
        0 & I+\varpi R^{-1}AR
    \end{pmatrix}     \\
    & = \psi_0(\varpi~\tr (R\tilde{\mathcal{B}_{\delta}}R^{-1}A)).
\end{align*} 
The last equality follows from the definition of $\phi_x$. \qedhere
\end{enumerate} 
\end{proof}

\begin{proposition}\label{Charcters at K12}
For $\delta\in \Omega$ with $\pi^{\delta}_{N,\psi}\neq 0$ and $A\in M_2(\cO_2)$, we have
$$\Theta_{\pi^{\delta}_{N,\psi}}(I+\varpi A)=\dfrac{q^{12}}{|\delta^{-1}T\delta\cap P|}\underset{R\in \GL_2(\F_q)}{\sum}\psi_0(\tr(R\mathcal{B}_{\delta}R^{-1}\bar{A})).$$
\end{proposition}

\begin{proof}
Using the sum over $X$, over $Q$ and over $(S,R)$ from Lemma \ref{S,X,R} in Equation (\ref{equation 1}), we get
\\
$$\Theta_{\pi^{\delta}_{N,\psi}}(I+\varpi A)=\frac{q^{12}}{q^8|\delta^{-1}T\delta\cap P|}\sum_{\substack{S,R\in \GL_2(\cO_2)\\
S=R\tilde{L}_{\delta}}}\psi_0(\varpi \tr(R\tilde{\mathcal{B}_{\delta}}R^{-1}A)).$$
The above sum does not change if we replace $R$ by $R'$ and $S$ by $S'$ such that $R=R'~(\modulo~\varpi)$ and $S=S'~(\modulo~\varpi)$. Therefore, we get
\begin{equation*}
\Theta_{\pi^{\delta}_{N,\psi}}(I+\varpi A)=\frac{q^{12}\cdot q^8}{q^8|\delta^{-1}T\delta\cap P|}\sum_{\substack{R\in \GL_2(\F_q)}}\psi_0(\tr(R\mathcal{B}_{\delta}R^{-1}\bar{A})). \qedhere
\end{equation*} 
\end{proof}

\begin{remark}\label{Intersection cardinality}
Note that $|\delta^{-1} T \delta \cap P| = \begin{cases} (q^2-1)q^{12} & \text{~if~} \delta=I, \\ (q-1)q^{12} & \text{~if~} \delta \neq I. \end{cases}
 $
\end{remark}

We denote $A\sim B$ if the two matrices $A$ and $B$ are conjugate. 
Recall that  any character of $M_2(\F_q)$ is of the form $\phi_B$ for $B\in M_2(\F_q)$. 
Since $M_2(\F_q) \cong J^{1}_{2}$ we write character of $J^{1}_{2}$ also as $\phi_{B}$.
\begin{corollary} \label{character phi_B} 
Consider $\delta \in \Omega$ such that $\pi^{\delta}_{N, \psi} \neq 0$. 
\begin{enumerate} [label = {(\alph*)}]
\item $\phi_{B}$ appears in $\pi^{\delta}_{N, \psi}|_{J^{1}_{2}}$ if and only if $B \sim \mathcal{B}_{\delta}$.
\item The multiplicity of $\phi_{\mathcal{B}_{\delta}}$ in $\pi^{\delta}_{N, \psi}$ is given by
\begin{center}
$\left\langle \Theta_{\pi^{\delta}_{N,\psi}},\phi_{\mathcal{B}_{\delta}}\right\rangle_{J^1_2}=$
$\begin{cases}
\frac{1}{q^2-1}\cdot |\Stab(2X_1)| & \textrm{~if~}\delta = I,
\\
\frac{1}{q-1}\cdot |\Stab(\mathcal{B}_\delta)| & \text{~if~}\delta\neq I 
\end{cases}$
\end{center}
where $\Stab(B)$ denotes the stabiliser of $B$ under the conjugation action of $\GL_2(\F_q)$ on $M_2(\F_q)$.
\end{enumerate}
\end{corollary}
\begin{proof}
Using Proposition \ref{Charcters at K12},  we get
\begin{equation}
\langle \Theta_{\pi^{\delta}_{N,\psi}},\phi_B\rangle_{J^1_2} 
= \dfrac{q^{12}}{|\delta^{-1}T\delta\cap P|}\cdot \dfrac{1}{q^4}\sum\limits_{\substack{R\in \GL_2(\F_q) \\
A\in M_2(\F_q)}}\psi_0(\tr(R\mathcal{B}_{\delta}R^{-1}-B)A). \label{equation 5.3}
\end{equation}
Part $(a)$ follows by observing that 
\begin{equation*}
\frac{1}{q^4}\cdot\sum\limits_{A\in M_2(\F_q)} \psi_0(\tr(R\mathcal{B}_{\delta}R^{-1}-B)A) = \left\{ \begin{array}{ll} 0 & \text{if~} R\mathcal{B}_{\delta}R^{-1} \neq B, \\
1 & \text{if~} R\mathcal{B}_{\delta}R^{-1} = B. 
\end{array} \right. \qedhere
\end{equation*}
Part $(b)$ follows from Remark \ref{Intersection cardinality} and Equation (\ref{equation 5.3}).
\end{proof}
\begin{remark}\label{Central character}
\begin{enumerate}[label = {(\alph*)}]
\item Under the considered embedding $\GL_2(\cO_2) \hookrightarrow \GL_4(\cO_2)$ the image of the center of $\GL_2(\cO_2)$ is equal to the center of $\GL_4(\cO_2)$ which are isomorphic to $\cO_{2}^{\times}$. We write $Z$ for both the center of $\GL_2(\cO_2)$ as well as that of $\GL_4(\cO_2)$.
\item Let the central character of representation $\pi$ be $\omega_\pi$ i.e. $\Theta_{\pi}(z)=\omega_{\pi}(z)\dim(\pi)=\tilde{\phi}_x(z)\dim(\pi)$ for $z \in Z$. 
It can easily be seen that for $\delta\in\Omega$, if $\pi^{\delta}_{N,\psi} \neq 0$ then it has a central character which is equal to $\omega_{\pi}$.
\end{enumerate}
\end{remark}
Proposition \ref{Charcters at K12} describes the character of $\pi^{\delta}_{N,\psi}$ at elements of $J^{1}_{2}$ and now we describe the character of $\pi^{\delta}_{N,\psi}$ on remaining elements of $\GL_2(\cO_2)$.
\begin{theorem}\label{Character values2}
Let $\delta \in \Omega$ be such that $\pi^{\delta}_{N,\psi} \neq 0$. Then,
\begin{enumerate}[label = {(\alph*)}]
\item For $z\in Z$ and $k\in J^1_2$, we have $\Theta_{\pi^{\delta}_{N,\psi}}(z\cdot k)=\omega_{\pi}(z)\Theta_{\pi^{\delta}_{N,\psi}}(k)$ . 
In particular,
$\Theta_{\pi^{\delta}_{N,\psi}}(z)=\omega_{\pi}(z)\cdot \dim(\pi^{\delta}_{N,\psi})$.
\item Let $\delta\neq I$. 
For $g\notin Z\cdot J^1_2$, we have  $\Theta_{\pi^{\delta}_{N,\psi}}(g)=0$.
\item If $g$ is not conjugate to any element in $\mathfrak{O}_2^\times\cdot J^1_2$, then $\Theta_{\pi^{I}_{N,\psi}}(g)=0$ .
\item
For $g\in \left(\mathfrak{O}_2^\times\cdot J^1_2\right)\backslash \left(Z\cdot J^1_2\right)$, we have
\begin{center}
$\Theta_{\pi^{I}_{N,\psi}}(g)=\dfrac{1}{q^4(q^2-1)}$\Bigg{[}$\underset{R\in \mathbf{\mathsf{N}}(\mathfrak{O}_2^\times\cdot J^1_2)}{\sum}\tilde{\phi}_x\begin{pmatrix}
R^{-1}gR & 0\\
0 & R^{-1}gR
\end{pmatrix}$\Bigg{]} 
\end{center}
where $\mathbf{\mathsf{N}}(\mathfrak{O}_2^\times\cdot J^1_2)$ is the normaliser of $\mathfrak{O}_2^\times\cdot J^1_2$ in $\GL_2(\cO_2)$.
\end{enumerate}
\end{theorem}

\begin{proof}
Since $\pi^{\delta} = \Ind_{\delta^{-1} T\delta\cap P}^{P}\tilde{\phi}_x^{\delta^{-1}}$, we get

\begin{equation} \label{greens formula applied}
\Theta_{\pi^{\delta}}\begin{pmatrix}
g & X\\
0 & g
\end{pmatrix}
= \dfrac{1}{|\delta^{-1} T\delta\cap P|} \sum\limits_{\gamma \in P}
\tilde{\phi}_x^{\delta^{-1}}\Big{(}\gamma^{-1}\begin{pmatrix}
g & X\\
0 & g
\end{pmatrix}\gamma\Big{)}
\end{equation}
where the summation is over $\gamma=\begin{pmatrix}
    S & Q\\
    0 & R
\end{pmatrix} \in P$ is such that 
\begin{equation}\label{conjugation with gamma}
\gamma^{-1}\begin{pmatrix}
    g & X\\
    0 & g
\end{pmatrix}\gamma=\begin{pmatrix}
S^{-1}gS & S^{-1}gQ+S^{-1}XR-S^{-1}QR^{-1}gR\\
0 & R^{-1}gR
\end{pmatrix} \in \delta^{-1} T \delta \cap P. 
\end{equation}

\begin{enumerate}[label = {(\alph*)}]
\item It follows from the Remark \ref{Central character} (b).

\item 
For $\delta\neq I$, 
$\delta^{-1}T\delta\cap P= Z\cdot (K^1_2\cap P)$.
By Equation (\ref{conjugation with gamma}), $\text{~both~}S^{-1}gS$ and $R^{-1}gR \in Z\cdot J^1_2$.
Therefore, if $g$ is not conjugate to $Z \cdot J^{1}_{2}$ (equivalently, $g \notin Z \cdot J^{1}_{2}$) then there does not exist any $\gamma \in P$ such that $\gamma^{-1}\begin{pmatrix}
g & X\\
0 & g
\end{pmatrix}\gamma\in (\delta^{-1} T\delta \cap P)$ and hence $\Theta_{\pi^{\delta}} \begin{pmatrix}
    g & X \\ 0 & g
\end{pmatrix} =0$.
By Corollary \ref{Greens formula}, we conclude $\Theta_{\pi^{\delta}_{N, \psi}}(g)=0$.
\item 
For $\delta=I$, $T\cap P = \Delta(\mathfrak{O}_2^\times)\cdot (K^1_2\cap P)$ where $\Delta(\mathfrak{O}_2^\times)$ is the diagonal embedding of $\mathfrak{O}_2^\times$ in $\GL_4(\cO_2)$. 
By Equation (\ref{conjugation with gamma}), $\text{~both~}S^{-1}gS$ and $R^{-1}gR \in \mathfrak{O}_2^\times\cdot J^1_2$. 
Therefore, if $g$ is not conjugate to any element in $\mathfrak{O}_2^\times\cdot J^1_2$, then $\Theta_{\pi^{I}} \begin{pmatrix}
    g & X \\ 0 & g
\end{pmatrix} =0$. 
By Corollary \ref{Greens formula}, we conclude $\Theta_{\pi^{I}_{N, \psi}}(g)=0$.
\item 
Using Equation (\ref{conjugation with gamma}), if $\gamma^{-1}\begin{pmatrix}
    g & X\\
    0 & g
\end{pmatrix}\gamma \in \delta^{-1}T\delta \cap P$ then, 
$$(S^{-1}gQ+S^{-1}XR-S^{-1}QR^{-1}gR)\in \varpi M_2(\cO_2).$$
\\
For $\delta=I$ and $g\in (\mathfrak{O}_2^\times\cdot J^1_2)\backslash (Z\cdot J^1_2)$, by Corollary \ref{Greens formula}, $\Theta_{\pi^{I}_{N,\psi}}(g)$ equals
{\small{\begin{align} \label{Formula for character value of trivial coset}
\dfrac{1}{|N|\cdot|T\cap P|}\underset{(S,R,Q,X)\in \mathcal{J}(g)}{\sum}\tilde{\phi}_x\begin{pmatrix}
S^{-1}gS & S^{-1}gQ+S^{-1}XR-S^{-1}QR^{-1}gR\\
0 & R^{-1}gR
\end{pmatrix}\overline{\psi}(Xg^{-1})
\end{align}}} 
where 
\\
{\small{$\mathcal{J}(g)=\left\lbrace (S,R,Q,X) : 
\begin{array}{ll} 
S, R \in \GL_2(\cO_2) \text{~and~} Q, X \in M_2(\cO_2), \\
S^{-1}gS,~R^{-1}gR\in \mathfrak{O}_2^\times \cdot J^1_2 
\text{~with~} S^{-1}gS=R^{-1}gR~(\modulo ~\varpi), \\  
S^{-1}gQ+S^{-1}XR-S^{-1}QR^{-1}gR\in\varpi M_2(\mathfrak{o}_2)  
\end{array}
\right\rbrace$.}}
\\
\\
The set $\mathcal{J}(g)$ describes the conditions on $\gamma$ and $X$ such that the matrix in (\ref{conjugation with gamma}) belongs to $T \cap P$.
 
Note that
\begin{align} \label{inner computation for J(g)}
& S^{-1}gQ+S^{-1}XR-S^{-1}QR^{-1}gR=0~(\modulo~\varpi) \nonumber\\  
\iff & gQR^{-1}+X-QR^{-1}g=0~(\modulo~\varpi) \nonumber \\
\iff & X=Q'g - g Q'~(\modulo~\varpi),
\end{align}
where $Q'=QR^{-1}$.
If $Q'_1$ and $Q'_2$ are two solutions of  $X=Q'g-gQ'~(\modulo~\varpi)$, then $(Q'_1-Q'_2)~(\modulo ~\varpi)$ belongs to the centraliser of $g~(\modulo~\varpi)\in\F_{q^2}\backslash\F_{q}$.
This gives $\mathcal{J}(g)$ consists of tuples $(S,R,Q,X)$ such that
\begin{enumerate}[label = {(\roman*)}]
\item $S,R \in \mathbf{\mathsf{N}}(\mathfrak{O}_2^\times\cdot J^1_2)$ with $RS^{-1}\in \mathfrak{O}_{2}^{\times} \cdot J^{1}_{2}$.
\item $Q'$ belongs to the homogeneous space of centraliser of $g~(\modulo~\varpi)$.
\item $X~(\modulo~\varpi)$ belongs to $\begin{pmatrix}
    1 & 0\\
    0 & -1
\end{pmatrix}\cdot \F_{q^2}$.
\end{enumerate}
Hence, 
{\small{\begin{equation}\label{cardinality of X,Q}
\left|\{Q\in M_2(\cO_2):(S,R,Q,X)\in \mathcal{J}(g)\}\right|=q^6=\left|\{X\in M_2(\cO_2):(S,R,Q,X)\in \mathcal{J}(g)\}\right|.  
\end{equation}}}
From Equation (\ref{Formula for character value of trivial coset}), $\Theta_{\pi^{I}_{N,\psi}}(g)$ is given by
\footnotesize{\begin{align*}\label{theorem 5.7 equation}
\dfrac{1}{|N||T\cap P|}\underset{(S,R,Q,X)\in \mathcal{J}(g)}{\sum} & \tilde{\phi}_x\begin{pmatrix}
S^{-1}gS & 0\\
0 & R^{-1}gR
\end{pmatrix}\tilde{\phi}_x\begin{pmatrix}
    I & S^{-1}Q+S^{-1}g^{-1}XR-S^{-1}g^{-1}QR^{-1}gR\\
    0 & I
\end{pmatrix} \overline{\psi}(Xg^{-1}).
\end{align*}}
\end{enumerate}
The theorem follows easily from the following two lemmas, which can be verified by using the non-degeneracy of the trace pairing and using $R^{-1} gR (\modulo~\varpi)$, $RS^{-1} (\modulo~\varpi)$ and $X_3 \in \F_{q^2}^{\times}$. 
For more details, look at the PhD Thesis of first author \cite{AnkitaThesis}.
\end{proof}
\begin{lemma}\label{lemma for theorem 5.7}
For fixed $R,S \in \GL_2(\cO_2)$ such that $RS^{-1} \in \mathfrak{O}_{2}^{\times} \cdot J^{1}_{2}$, we have
\begin{equation*}\label{Reduction of X and Q}
\sum_{\substack{X,Q \in M_2(\cO_2) ~\text{s.t.}\\(S,R,Q,X) \in \mathcal{J}(g) }} \tilde{\phi}_x\begin{pmatrix}
I & S^{-1}Q+S^{-1}g^{-1}XR-S^{-1}g^{-1}QR^{-1}gR\\
0 & I
\end{pmatrix}\overline{\psi}(Xg^{-1})= q^{12}.
\end{equation*}
\end{lemma}
\begin{lemma}\label{character value is independent of S}
Let $S,R\in \mathbf{\mathsf{N}}(\mathfrak{O}_2^\times\cdot J^1_2)$ such that $SR^{-1}\in \mathfrak{O}_2^\times\cdot J^1_2$ then, for any $g \in \mathfrak{O}_2^\times\cdot J^1_2$ we have
$$\tilde{\phi}_x\begin{pmatrix}
    S^{-1}gS & 0\\
    0 & R^{-1}gR
\end{pmatrix}=\tilde{\phi}_x\begin{pmatrix}
    R^{-1}gR & 0\\
    0 & R^{-1}gR
\end{pmatrix}.$$
\end{lemma}

\section{A description of $\pi_{N, \psi}$}\label{Sec : 6}
Since $\pi_{N, \psi} = \underset{\delta \in \Omega_{0}}{\bigoplus}\pi^{\delta}_{N, \psi}$, we describe all the irreducible representations of $\GL_2(\cO_2)$ appearing in $\pi^{\delta}_{N, \psi}$ for all $\delta \in \Omega$. 

\begin{proposition}\label{pi-delta-N,psi contains only regular representation}
Let $\pi$ be a strongly cuspidal representation of $\GL_4(\cO_2)$.
For $\delta \in \Omega$ with $\delta \neq I$ and an irreducible representation $\sigma$ of $\GL_2(\cO_2)$, if $\sigma$ appears in $\pi^{\delta}_{N,\psi}$ then $\sigma$  is a regular representation.
\end{proposition} 
\begin{proof}
Let $B \in M_{2}(\F_{q})$ be such that $\left\langle\Theta_{\sigma},\phi_{B}\right\rangle_{J^1_2}\neq 0$.
Since $\sigma $ appears in $\pi^{\delta}_{N,\psi}$ then $\phi_{B}$ appears in $\pi^{\delta}_{N, \psi}|_{J^{1}_{2}} $.
By Corollary \ref{character phi_B}, $B\sim \mathcal{B}_{\delta}$ and,
therefore it is enough to prove that $\mathcal{B}_{\delta}$ is not a scalar matrix. 

 Now we assume that $\mathcal{B}_{\delta}$ is a scalar matrix and deduce a contradiction.
Recall that $\tr(\mathcal{B}_{\delta}) = \tr(2X_1)$, then 
$\mathcal{B}_{\delta}=\begin{pmatrix}
     2(a_0+aa_1) & 0\\
    0 & 2(a_0+aa_1)
\end{pmatrix}.$ 
Write
$$\mathcal{B}_{\delta} = L_{\delta}(X_1+X_2Y)L_{\delta}^{-1}+X_1-YX_2 =L_{\delta}\left(X_1+X_2Y+L_{\delta}^{-1}(X_1-YX_2)L_{\delta}\right)L_{\delta}^{-1}.$$ 
This gives $X_1+X_2Y+L_{\delta}^{-1}(X_1-YX_2)L_{\delta}$ is diagonalizable and both the columns of $L_{\delta}$ are its eigenvectors with eigenvalue $2(a_0+aa_1)$. Write 
$$L_{\delta} 
=\begin{pmatrix}
    L_{\delta}^1 &  L_{\delta}^2\\
   \end{pmatrix},$$ 
where $L_{\delta}^1$ and $L_{\delta}^1$ are the columns of the matrix $L_{\delta}$. 
Therefore, 
\begin{align}
\left(X_1+X_2Y+L_{\delta}^{-1}(X_1-YX_2)L_{\delta}\right)\cdot L_{\delta}^1=2(a_0+aa_1)\cdot L_{\delta}^1. \label{6.1} \\ 
\left(X_1+X_2Y+L_{\delta}^{-1}(X_1-YX_2)L_{\delta}\right)\cdot L_{\delta}^2=2(a_0+aa_1)\cdot L_{\delta}^2.\label{6.2}
\end{align}
Equation (\ref{6.1}) gives
\begin{equation}\label{values of u and v}
u=\dfrac{-2a_1(a_3(a^2-b^2\alpha)+aa_2)}{a_2^2+(3a^2+b^2\alpha)a_3^2+4aa_2a_3}\text{ and }v=\dfrac{2a_1(a_2+aa_3)}{a_2^2+(3a^2+b^2\alpha)a_3^2+4aa_2a_3}.
\end{equation}
Equation (\ref{6.2}) gives two quadratic equations 
\begin{equation} \label{Q_1 Q_2}
Q_1(u,v)=0 \text{~~and~~} Q_2(u,v)=0.
\end{equation}

If $a_1=0$, then Equation (\ref{values of u and v}) gives $u=v=0$ and hence $\delta =I$, which contradicts our assumption $\delta \neq I$.
\\

Now we assume that $a_1\neq 0$. 
Substituting the values of $u,v$ from Equation (\ref{values of u and v}) in Equations (\ref{Q_1 Q_2}), one arrives at a contradiction to our assumption $x \in \F_{q^4} \backslash \F_{q^2}$, for which we skip the details.
\end{proof}

\begin{proposition}\label{multplicity 1 of sigma}
Let $\pi$ be a strongly cuspidal representation of $\GL_4(\cO_2)$. For $\delta \in \Omega$ and $\delta \neq I$, the representation $\pi^{\delta}_{N, \psi}$ is multiplicity free. 
In other words, if $\sigma$ is a regular representation of $\GL_2(\mathfrak{o}_2)$ such that $\omega_{\pi}=\omega_{\sigma}$ and
$\phi_{\mathcal{B}_{\delta}}$ appears in $\sigma|_{J^{1}_{2}}$, then $\sigma$ appears in $\pi^{\delta}_{N, \psi}$ with multiplicity one.
\end{proposition}
\begin{proof}
For $\delta \in \Omega$ and $\delta \neq I$, by 
Theorem \ref{Character values2} (b), we have
$\Theta_{\pi^{\delta}_{N,\psi}}(g)=0$ if $g\notin Z\cdot J^1_2$. 
Using $\omega_{\pi} = \omega_{\sigma}$ we get,
\begin{align*}
\left\langle \Theta_{\pi^{\delta}_{N,\psi}}, \Theta_{\sigma} \right\rangle_{\GL_2(\mathfrak{o}_2)} 
& = \dfrac{1}{|\GL_2(\cO_2)|}\underset{g=z\cdot k\in Z\cdot J^1_2}{\sum}\Theta_{\pi^{\delta}_{N,\psi}}(z\cdot k)\overline{\Theta}_{\sigma}(z\cdot k) \\
& = \dfrac{|Z/(J^1_2\cap Z)|}{|\GL_2(\cO_2)|}\underset{k\in J^1_2}{\sum}\Theta_{\pi^{\delta}_{N,\psi}}(k)\overline{\Theta}_{\sigma}(k) \\
& = \dfrac{|Z/(J^1_2\cap Z)|\cdot |J^1_2|}{|\GL_2(\cO_2)|}\left\langle \Theta_{\pi^{\delta}_{N,\psi}},\Theta_{\sigma}\right\rangle_{J^1_2}.
\end{align*}
By Corollary \ref{character phi_B}, we have
$\pi^{\delta}_{N,\psi}|_{J^{1}_{2}} 
= \underset{B \sim \mathcal{B}_{\delta}}{\bigoplus} \left( \dfrac{|\Stab(B)|}{q-1} \right) \phi_B$ 
and by Remark \ref{property of regular representation}, it follows that
$ \sigma|_{J^{1}_{2}} = \underset{B \sim \mathcal{B}_{\delta}}{\bigoplus}\phi_{B}$.
Therefore, 
$$\left\langle \Theta_{\pi^{\delta}_{N,\psi}},\Theta_{\sigma}\right\rangle_{J^1_2}
= \left( \dfrac{|\Stab(\mathcal{B}_{\delta})|}{q-1} \right) \cdot  \left|\text{Conjugacy class of $\mathcal{B}_{\delta}$}\right|,
$$ which gives
 $
\left\langle \Theta_{\pi^{\delta}_{N,\psi}}, \Theta_{\sigma}\right\rangle_{\GL_2(\mathfrak{o}_2)}=1.$
\end{proof}
\begin{remark}\label{kind of representation in pi delta N psi}
Note that for $\delta\neq I$, the number of $\sigma$ satisfying the property in 
Proposition \ref{multplicity 1 of sigma} is $\frac{|\Stab(\mathcal{B}_{\delta})|}{q-1}$ and each having dimension $\frac{|\GL_2(\F_q)|}{|\Stab(\mathcal{B}_{\delta})|}$.
This gives  $\dim(\pi^{\delta}_{N,\psi})=q(q^2-1) = \left( \frac{|\Stab(\mathcal{B}_{\delta})|}{q-1} \right) \cdot \left( \frac{|\GL_2(\F_q)|}{|\Stab(\mathcal{B}_{\delta})|} \right)$. 
Therefore, an irreducible representation $\sigma$ appears in $\pi^{\delta}_{N,\psi}$ if and only if $\sigma$ satisfies the properties mentioned in Proposition \ref{multplicity 1 of sigma}.
\end{remark}
\begin{lemma} \label{pi^1_{N,psi}}
The representation $\pi^{I}_{N,\psi} \cong \Ind_{\mathfrak{O}_2^\times\cdot J^1_2}^{\GL_2(\mathfrak{o}_2)}\left( \tilde{\phi}_{x}|_{\mathfrak{O}_2^\times\cdot J^1_2}\right)$. 
\end{lemma}
\begin{proof}
It is enough to verify that the characters of the two representations are the same.
The character of $\pi^{I}_{N, \psi}$ is given in Theorem \ref{Character values2}. 
The character of  $ \Ind_{\mathfrak{O}_2^\times\cdot J^1_2}^{\GL_2(\mathfrak{o}_2)}\left( \tilde{\phi}_{x}|_{\mathfrak{O}_2^\times\cdot J^1_2}\right)$ can be calculated using the formula for the character of an induced representation.
We leave the verification to the reader that these two characters are the same.
\end{proof}
\begin{corollary} \label{Pi I N psi is multiplicity free}
The representation $\pi^{I}_{N,\psi}$  is multiplicity free.  
\end{corollary}
\begin{proof}
First consider $X_1\in \F_{q^2}\backslash\F_q$. 
Then $\pi^{I}_{N,\psi} $ is an irreducible regular representation of non-split semisimple type and hence nothing to prove.
\\
Now consider $X_1\in \F_q$. 
By Clifford theory, there exists a character $\zeta:\GL_2(\cO_2) \rightarrow \C^{\times}$ such that $\zeta|_{J^1_2} =\phi_{-(2X_1)}$.
Then $\theta' := (\theta \cdot \phi_{2X_1}) \cdot \zeta|_{\mathfrak{O}_2^\times\cdot J^1_2}$ is trivial on $J^{1}_{2}$ and therefore $\theta'$ factors through $(\mathfrak{O}_2^\times\cdot J^1_2) / J^{1}_{2} \cong \F_{q^2}^{\times}$. 
Since $\pi^{I}_{N,\psi}\cong \Ind_{\mathfrak{O}_2^\times\cdot J^1_2}^{\GL_2(\cO_2)} \left( \theta\cdot \phi_{2X_1} \right)$, we get $\pi^{I}_{N,\psi}\otimes~\zeta\cong \Ind_{\mathfrak{O}_2^\times\cdot J^1_2}^{\GL_2(\cO_2)} \left( \theta' \right)$ (see \cite[Exercise 3.16]{Fulton-Harris}). 
It follows that $\pi^{I}_{N,\psi}\otimes~\zeta$ is trivial on $J^1_2$ and hence factors through $\GL_2(\cO_2)/ J^1_2 \cong \GL_2(\F_q)$.
It is enough to prove that $\pi^{I}_{N, \psi} \otimes \zeta \cong \Ind_{\F_{q^2}^{\times}}^{\GL_2(\F_q)} (\theta')$ is multiplicity free representation as a representation of $\GL_2(\F_q)$.
It follows from the fact that $(\GL_2(\F_q), \F_{q^2}^{\times})$ is a strong Gelfand pair which can be easily verified by looking at the character table of $\GL_2(\F_q)$.   
\end{proof}

\begin{theorem}\label{Matrices conjugate iff cosets are same}
For $\delta, \delta' \in \Omega$, $\mathcal{B}_\delta \sim \mathcal{B}_{\delta^\prime}$ if and only if $T\delta P=T\delta^\prime P$.
\end{theorem}
\begin{proof}
First, assume that $T\delta P=T\delta^{\prime}P$. 
Then $\pi^{\delta}_{N, \psi} = \pi^{\delta'}_{N, \psi}$ and by Corollary \ref{character phi_B}, we get $\mathcal{B}_\delta\sim \mathcal{B}_{\delta^\prime}$. 
For the converse, assume that $\mathcal{B}_\delta\sim \mathcal{B}_{\delta^\prime}$. 
Then, $\det(\mathcal{B}_\delta)=\det( \mathcal{B}_{\delta^\prime})$. 
It can be checked that for $\delta =A_{u,v}$, the determinant of $\mathcal{B}_{\delta}$ is given by
$$\det(\mathcal{B}_{\delta})=4(a_0+aa_1)^2-\left[\dfrac{4a_1^2(A)
-\left((u+av)^2-(bv)^2\alpha\right)\cdot\left(a_2^2+a_3^2(3a^2+b^2\alpha)+4aa_2a_3\right)^2 }{\det(L_{\delta})}\right]$$
where $A=A(u,v)$ is a quadratic polynomial in $u$ and $v$. \\
Moreover, for $\delta=A_{u,v}$ and $\delta' = A_{u', v'}$ it can be deduced that the conditions for $\det(\mathcal{B}_{\delta}) = \det( \mathcal{B}_{\delta'})$ comes out to be the following for which we refer the reader to \cite{AnkitaThesis} and omit the details
\begin{align}\label{Equation for $M$}
& M \cdot (a^2a_3^2+2aa_2a_3+a_2^2-a_3^2b^2\alpha)\cdot C(\delta, \delta') =0,
\end{align}
where 
{\small{\begin{align*}
M = 4a_1^4b^2\alpha  - 4a_1^2(a^3a_3^2+2a^2a_2a_3+aa_2^2+3aa_3^2b^2\alpha+2a_2a_3b^2\alpha)
 +(a_2^2 +a_3^2(3a^2+ b^2\alpha)+4aa_2a_3)^2.
\end{align*}}}
Observe that the second factor in (\ref{Equation for $M$}) is not zero, since $$a^2a_3^2+2aa_2a_3+a_2^2-a_3^2b^2\alpha = (a_2+aa_3)^2-(a_3b)^2\alpha = 0 \implies a_2=0, a_3=0$$ 
and then $x \in \F_{q^2}$, a contradiction.\\
Now we claim that $M\neq 0$.
Note that $M$ is a quadratic polynomial in $a_1^2$.
If $M=0$, then the discriminant of this quadratic polynomial is a square in $\F_{q}$.
But, the discriminant is given by
$$16(a^2-b^2\alpha)\left((a_2+aa_3)^2-(a_3b)^2\alpha\right)^2
$$
which is not a square in $\F_q$ (see the assumption (\ref{Norm is square free})).
This proves $M \neq 0$.\\
Therefore, by Equation (\ref{Equation for $M$}) we get $C(\delta, \delta')=0$ and Lemma \ref{Double cosets} gives  $T\delta P = T\delta'P$.
\end{proof}
We have described each $\pi^{\delta}_{N, \psi}$ and proved that for distinct $\delta, \delta' \in \Omega_{0}$ the restriction of $\pi^{\delta}_{N, \psi}$ and $\pi^{\delta'}_{N, \psi}$ to $J^{1}_{2}$ are distinct, which gives the following result. 
\begin{corollary} \label{disjointness of pi^delta}
For $\delta, \delta' \in \Omega_{0}$ with $\delta \neq \delta'$ the representations $\pi^{\delta}_{N, \psi}$ and $\pi^{\delta'}_{N, \psi}$ have no irreducible representation in common.
\end{corollary}

\begin{theorem}
Let $\pi$ be a strongly cuspidal representation of $\GL_4(\cO_2)$. 
Then $\pi_{N,\psi}$ is multiplicity free.
\end{theorem}
\begin{proof}
For $\delta, \delta' \in \Omega_{0}$ such that $\delta \neq \delta'$, by Corollary \ref{disjointness of pi^delta}, $\pi^{\delta}_{N,\psi}$ and $\pi^{\delta'}_{N,\psi}$ have no irreducible representation in common.
By Proposition \ref{multplicity 1 of sigma} and Corollary \ref{Pi I N psi is multiplicity free} $\pi^{\delta}_{N, \psi}$ is multiplicity free for all $\delta \in \Omega_{0}$ and hence $\pi_{N, \psi}$ is multiplicity free.
\end{proof}

Recall that $| \Omega_{0}| =q+1$ (Lemma \ref{no of distinct double cosets}) and $|\{ \delta \in \Omega_{0} : \pi^{\delta}_{N, \psi} \neq 0 \}| =q$ (Lemma \ref{Uniqueness}). 
Therefore, there are exactly $q$ conjugacy classes represented by $\mathcal{B}_{\delta}$ such that $\phi_{\mathcal{B}_{\delta}}$ appears in $\pi_{N, \psi}$ restricted to $K^{1}_{2}$. 
Moreover, all the $\mathcal{B}_{\delta}$'s  have distinct determinant (Theorem \ref{Matrices conjugate iff cosets are same}) and then, 
$$\{ \det(\mathcal{B}_{\delta}) : \delta \in \Omega_{0} \} = \F_{q}.$$
Note that the trace of each $\mathcal{B}_{\delta}$ is the same as $\tr(2X_{1})$. 
We conclude the following :
\\
(1) all the $q-1$ semisimple (both split and non-split) conjugacy classes with trace $\tr({2X_1})$ are represented by $\mathcal{B}_{\delta}$ for some $\delta \in \Omega_{0}$. \\
(2) The determinant of the scalar matrix and the split non-semisimple matrices with trace $\tr(2X_1)$ are the same and exactly one of these two conjugacy classes will be represented by $\mathcal{B}_{\delta}$ for some $\delta \in \Omega_{0}$.

We summarise the results discussed in Section \ref{Sec : 6} in the theorem below.

\begin{theorem} \label{Complete description of pi_N,psi}
Let $\pi=\Ind_{I(\phi_x)}^{\GL_4(\cO_2)}\tilde{\phi}_x$ be a strongly cuspidal representation of $\GL_4(\cO_2)$. 
\begin{enumerate}[label = {(\alph*)}]
\item Let $X_1\in \F_q$, then $\pi_{N,\psi}$ consists of the following representations of $\GL_2(\cO_2)$.
\begin{enumerate}[label = {(\roman*)}]
\item  The representation $\pi^{I}_{N,\psi} = \Ind_{\mathfrak{O}_2^\times \cdot J^{1}_{2}}^{\GL_2(\cO_2)} \left( \tilde{\phi}_{2X_1} \right)$, which is not a regular representation and it is also not irreducible.
\item An irreducible regular representation $\sigma$ of $\GL_2(\cO_2)$ with multiplicity one such that $\omega_{\sigma}=\omega_{\pi}$ and $\sigma|_{J^{1}_{2}}$ contains the character $\phi_B$ with $B$ a semisimple matrix of trace $\tr(2X_1)$. 
\end{enumerate}

\item Let $X_1\in \F_{q^2} \backslash\F_q$, then $\pi_{N,\psi}$ consists of the following representations of $\GL_2(\cO_2)$.
\begin{enumerate}[label = {(\roman*)}]
\item The irreducible regular representation $\pi^{I}_{N,\psi}=\Ind_{\mathfrak{O}_2^\times \cdot J^{1}_{2}}^{\GL_2(\cO_2)} \left( \tilde{\phi}_{2X_1} \right)$ with multiplicity one.
\item An irreducible regular representation $\sigma$ of $\GL_2(\cO_2)$ with multiplicity one such that $\omega_{\sigma}=\omega_{\pi}$ and $\sigma|_{J^{1}_{2}}$ contains the character $\phi_B$ with $B$ a regular matrix of trace $\tr(2X_1)$ and $B \nsim 2X_1$. 
\end{enumerate}
\end{enumerate}
\end{theorem}

\section{A description of $\Pi= \Ind_{\mathfrak{O}_2^\times}^{\GL_2(\mathfrak{o}_2)}(\theta|_{\mathfrak{O}_2^\times})$}\label{Sec : 7}

Recall from Section \ref{Sec:4} that we have fixed $a, b, \alpha \in \F_{q}$ such that $\alpha$ is a square free element in $\F_{q}$ and $a+b \sqrt{\alpha}$ is a square free element in $\F_{q^2}= \F_{q}(\sqrt{\alpha})$. We have also fixed an embedding $\F_{q^2} \hookrightarrow M_2(\F_q)$ as given in (\ref{F_q^2 embedding}).
We state the following easy lemma without proof. 
It provides a set of representatives for various double cosets, which will be useful later in this section.
\begin{lemma} \label{Various double cosets1}
For $y \in \F_{q}, z \in \F_{q}^{\times}$ write $\gamma(y,z) 
=\begin{pmatrix}
1 & y \\ 0 & z    
\end{pmatrix} $ and 
$$\Gamma = \left\lbrace \gamma(y,z) : y \in \F_{q}, z \in \F_{q}^{\times} \right\rbrace.$$

\begin{enumerate}[label = {(\alph*)}]
\item The set $\Gamma$ is a set of distinct representatives of $\F_{q^2}^\times\backslash\GL_2(\F_q)$.
\item For $\gamma(y,z)$ and $\gamma(y',z') \in \Gamma$ ,
$\mathbb{F}_{q^2}^\times~\gamma(y,z)~\mathbb{F}_{q^2}^\times=\mathbb{F}_{q^2}^\times~\gamma(y',z')~\mathbb{F}_{q^2}^\times$~~if and only if~~
\\
$$\dfrac{y^2+2ay(z-1)+(1+z^2)(a^2-b^2\alpha)}{z}=\dfrac{y'^2+2ay'(z'-1)+(1+z'^2)(a^2-b^2\alpha)}{z'}.$$ 
We fix a subset $\Gamma_0\subset\Gamma$ representing distinct  double cosets for $\F_{q^2}^\times\backslash\GL_2(\F_q)/\F_{q^2}^\times$.
\item Let $T_1 = \left\lbrace \begin{pmatrix}
    m & 0\\
    0 & m'
\end{pmatrix} : m, m' \in \F_q^\times \right\rbrace$. 
For $\F_{q^2}^{\times} \backslash \GL_2(\F_q)/T_1$, a set of distinct representatives  is given by
$\Gamma_{1} := \left\lbrace \gamma(y,1): y\in \F_q\right\rbrace$.
\item Let $T_2 = \left\lbrace \begin{pmatrix}
    m & n \\
    0 & m
\end{pmatrix} :m\in \F_q^\times,~ n\in \F_q  \right\rbrace$. 
For $\F_{q^2}^\times\backslash\GL_2(\F_q)/T_2$, a set of distinct representatives  is given by
$\Gamma_{2} : =\left\lbrace  \gamma(0,z) : z\in \F_q^\times \right\rbrace$.
\end{enumerate}
\end{lemma}

\begin{remark}
Let $H,K$ be subgroups of $G=\GL_2(\cO_2)$. 
Assume that the quotient map $G \rightarrow \bar{G}= \GL_2(\F_q)$ induces a bijection between the double cosets $K\backslash G/H$ and $\bar{K}\backslash \bar{G}/\bar{H}$, where $\bar{H}, \bar{K}$ denote the images of $H, K$ in $\bar{G}$.
Then, for a representative $\gamma\in \bar{G}$ of a double coset in $\bar{K}\backslash \bar{G}/\bar{H}$, we write its lift to $G$ by the same letter $\gamma$ as a representative of the corresponding double coset in $K\backslash G/H$.
\end{remark}
The following proposition describes the characters $\phi_{B}$ of $J^{1}_{2} \cong M_{2}(\F_q)$ which appear in the restriction of $\Pi$ to ${J^{1}_{2}}$.
\begin{proposition}\label{Matrix B in witt ring side}
For $B\in M_2(\F_q)$, $\phi_B$ appears in $\Pi|_{J^{1}_{2}}$ if and only if $\exists ~m, n \in \F_{q}$ such that
$$B\sim 2X_1-\begin{pmatrix}
    m & n(a^2-b^2\alpha)+2am\\
    n & -m
\end{pmatrix}.$$ 
\end{proposition}
\begin{proof}
Recall that the character $\phi_B:M_2(\F_q)\rightarrow\C^\times$ is given by $\phi_B(A)=\psi_0(\tr(BA))$. 
Since $\Pi$ is an induced representation, by Mackey theory we have $$\Pi|_{J^1_2}=\underset{\gamma\in \mathfrak{O}_2^\times\backslash \GL_2(\cO_2)/J^1_2}{\bigoplus}\Ind_{\gamma^{-1}\mathfrak{O}_2^\times\gamma \cap J^1_2}^{J^1_2}\left(\theta^{\gamma^{-1}}\right).$$ 
Using Frobenius reciprocity, we get 
\begin{align}
\left\langle \Theta_{\Pi},\phi_B\right\rangle_{J^1_2}
= \underset{\gamma\in \mathfrak{O}_2^\times\backslash \GL_2(\cO_2)/J^1_2}{\sum}\left\langle \theta^{\gamma^{-1}},\phi_{B}\right\rangle_{\gamma^{-1}(\mathfrak{O}_2^\times)\gamma\cap J^1_2}.
\end{align}
Since $J^{1}_{2}$ is a normal subgroup, $\gamma^{-1}(\mathfrak{O}_2^\times)\gamma\cap J^1_2 = \gamma^{-1} (\mathfrak{O}_2^\times \cap J^{1}_{2}) \gamma = \gamma^{-1} K \gamma$ where $K=\mathfrak{O}_2^\times \cap J^{1}_{2}$. 
Note that there is a bijection between $\mathfrak{O}_2^\times\backslash \GL_2(\cO_2)/J^1_2$ and $\F_{q^2}^\times\backslash \GL_2(\F_q)$. Using  the Lemma \ref{Various double cosets1} (a) we get, 
\begin{align}\label{equation 7.2}
\left\langle \Theta_{\Pi},\phi_B\right\rangle_{J^1_2} 
=\underset{\gamma\in \Gamma}{\sum}\left\langle \theta^{\gamma^{-1}},\phi_{B}\right\rangle_{\gamma^{-1}(K)\gamma} =\underset{\gamma\in \Gamma}{\sum}\left\langle \theta,\phi_{B}^{\gamma}\right\rangle_{K}.
\end{align}
Note that $\theta|_{K}=\phi_{2X_1}|_{K}$, which gives
\begin{align} \label{equation 1 from proposition 7.3}
\left\langle \theta,\phi_{B}^{\gamma}\right\rangle_{K}
=\left\langle \phi_{2X_1},\phi_{B}^{\gamma}\right\rangle_{K}
& =\dfrac{1}{|K|}\underset{A\in K}{\sum}\psi_0(\tr  (2X_1A))\overline{\psi_0}(\tr(\gamma B\gamma^{-1} A)) \nonumber \\
& =\dfrac{1}{|K|}\underset{A\in K}{\sum}\psi_0( \tr(2X_1-\gamma B\gamma^{-1})A).
\end{align}
We know that $\tr:M_2(\F_q)\times M_2(\F_q)\rightarrow \F_q$ is a non-degenerate bilinear map. 
Using the isomorphism $J^1_2\cong M_2(\F_q)$, $K$ is identified with $\left\lbrace\begin{pmatrix}
     b_0 & -b_1(a^2-b^2\alpha)\\
    b_1 & b_0+2ab_1
\end{pmatrix}~:~b_0,b_1\in \F_q\right\rbrace$.
It can be checked that
\begin{center}
$\left\lbrace C\in M_2(\F_q):\tr(CD)=0~\forall~D\in K\right\rbrace=\left\lbrace\begin{pmatrix}
    m & n(a^2-b^2\alpha)+2am\\
    n & -m
\end{pmatrix}:m,n\in \F_q\right\rbrace$.
\end{center}
Note that $\langle \theta,\phi_{B}^{\gamma}\rangle_{K}$ is either $0$ or $1$. 
Moreover, $\left\langle \theta,\phi_{B}^{\gamma}\right\rangle_{K} =1$ if and only if 
\begin{align}\label{two characters restricted to K}
 \exists ~m,n \in \F_{q} \text{~such that~} 2X_1-\gamma B\gamma^{-1} = \begin{pmatrix}
   m & n(a^2-b^2\alpha)+2am\\
    n & -m
\end{pmatrix}.
\end{align}
Therefore, $\phi_B$ appears in $\Pi|_{J^1_2}$ if and only if  $B$ satisfies (\ref{two characters restricted to K}) for some $\gamma \in \Gamma$ and $m,n \in \F_{q}$.
\end{proof}

\begin{corollary}\label{2X_1 multiplicity is one}
If $X_1\in \F_{q^2}\backslash\F_q$, then $\phi_{2X_1}$ appears in $\Pi|_{J^1_2}$ with multiplicity one.
\end{corollary}
\begin{proof}
Using Equation (\ref{equation 7.2}) , (\ref{equation 1 from proposition 7.3}) and (\ref{two characters restricted to K}) for $B=2X_1$, we have
\begin{align*}
\left\langle \Theta_{\Pi},\phi_{2X_1}\right\rangle_{J^1_2} & 
=\dfrac{1}{|K|}\underset{\gamma\in \Gamma,~A\in K}{\sum}\psi_0( \tr(2X_1- \gamma 2X_1\gamma^{-1})A)
\\
& = \left|\left\lbrace \gamma \in \Gamma:~2X_1-\gamma 2X_1\gamma^{-1}\in \left\lbrace \begin{pmatrix}
    m & n(a^2-b^2\alpha)+2am\\
    n & -m
\end{pmatrix}:m,n\in \F_q\right\rbrace\right\rbrace\right|
\\
& =\left|\left\lbrace\gamma(y,z) \in \Gamma :\frac{y^2+2ay(z-1)+(1+z^2)(a^2-b^2\alpha)}{z}=2(a^2-b^2\alpha)~\right\rbrace\right|
\\
& =1.\qedhere
\end{align*}
\end{proof}
\begin{lemma}\label{Matrices in Pi}
Let $B = 2X_1-\begin{pmatrix}
   m & n(a^2-b^2\alpha)+2am\\
    n & -m
\end{pmatrix}$ for some $m,n \in \F_{q}$.
\begin{enumerate}[label = {(\alph*)}]
\item  If $X_1\in \F_q$, then either $B=2X_1$ or $B$ is a semisimple matrix with trace $\tr(2X_1)$.
\item If $X_1 \in \F_{q^2} \backslash \F_q$, then $B$ is a regular matrix with trace $\tr(2X_1)$.
\end{enumerate}
\end{lemma}
\begin{proof}
It is clear that $\tr(B)= \tr(2X_1)$.
The characteristic polynomial of $B$ is given by
$$\left(\lambda-2(a_0+aa_1)\right)^2 -\left( 4a_1^2b^2\alpha + (m+an)^2 -b^2n^2 \alpha \right).$$ 
Clearly, if $D := 4a_1^2b^2\alpha + (m+an)^2 -b^2n^2 \alpha \neq 0$, then $B$ will be semisimple. 
If $D = 0$, then $B$ is either scalar or split non-semisimple with trace $\tr(2X_1)$.
\begin{enumerate}[label = {(\alph*)}]
\item Let $X_1 \in \F_q$ i.e. $a_1=0$. 
Then, $D=0 \implies  m=0, n=0 \implies B=2X_1$ which is scalar. 
\item Let $X_1 \in \F_{q^2}\backslash \F_{q}$, i.e. $a_1 \neq 0$.
If $D=0$ and $B$ is the scalar matrix with trace $\tr(2X_1)$ then by comparing the entries we get $m=-2aa_1, n=2a_1$ and $4a_1 b^{2} \alpha =0$.
Since $a_1 \neq 0$, $4a_1 b^2 \alpha =0$ is not possible, a contradiction.
In fact, the choice $m=-2aa_1, n=2a_1$ gives a split non-semisimple regular matrix. \qedhere 
\end{enumerate}
\end{proof}
Let $\phi_{B}$ be a character of $J^{1}_{2}$ and $\sigma = \Ind_{I(\phi_B)}^{\GL_2(\mathfrak{o}_2)}\tilde{\phi}_B$.
By Mackey theory, it follows that 
\begin{align}\label{Equation for Pi,sigma}
\left\langle \Theta_{\Pi},\Theta_{\sigma}\right\rangle_{\GL_2(\cO_2)} 
=\left\langle \theta, \Theta_{\sigma}\right\rangle_{\mathfrak{O}_2^\times}
& =\underset{\gamma\in \mathfrak{O}_2^\times\backslash \GL_2(\mathfrak{o}_2)/I(\phi_B)}{\sum} \left\langle \theta, \tilde{\phi}_{B}^{\gamma} \right\rangle_{\gamma \left(I(\phi_B)\right) \gamma^{-1}\cap \mathfrak{O}_2^\times}.
\end{align}

\begin{lemma}
\label{reduction of calculation for multiplicity 1}
Let $B \in M_{2}(\F_{q})$ be a non-split semisimple matrix with trace $\tr(2X_1)$. 
Let $\tilde{\phi}_{B} : I(\phi_{B}) \rightarrow \C^{\times}$ be an extension of $\phi_{B}$ and $\sigma = \Ind_{I(\phi_B)}^{\GL_2(\mathfrak{o}_2)}\tilde{\phi}_B$.
Then,
\begin{equation}\label{Inner product of characters in Pi}
\langle \Theta_{\Pi},\Theta_{\sigma}\rangle_{\GL_2(\cO_2)}=\left\langle\theta,\tilde{\phi}_{B} \right\rangle_{\mathfrak{O}_2^\times}+\left\langle\theta,\tilde{\phi}_{B}^{\gamma_0} \right\rangle_{\mathfrak{O}_2^\times}+\sum\limits_{\substack{\gamma\in \Gamma_0\text{~such that~}\\
\gamma\notin\{I,\gamma_0\}}}\left\langle\tilde{\phi}_x,\tilde{\phi}_{B}^{\gamma}\right\rangle_{Z\cdot K}.
\end{equation}
where $\gamma_{0} \in \Gamma_0$ is a non-trivial element in $\mathbf{\mathsf{N}}_{\GL_2(\F_q)}(\F_{q^2}^\times)$  the normalizer of $\F_{q^2}^{\times}$, and $K=I+\varpi \mathfrak{O}_2\cong \F_{q^2}$.
\end{lemma}
\begin{proof}
We use Equation (\ref{Equation for Pi,sigma}) to prove the lemma.
Note that $I(\phi_B)=\mathfrak{O}_2^\times \cdot J^{1}_{2}$ and if we write $K=I+\varpi \mathfrak{O}_2$, then
\begin{center}
$\gamma\left(\mathfrak{O}_2^\times\cdot J^1_2\right)\gamma^{-1}\cap \mathfrak{O}_2^\times
=\begin{cases}
\mathfrak{O}_2^\times & \text{~if~} \gamma\in \mathbf{\mathsf{N}}_{\GL_2(\cO_2)}(\mathfrak{O}_2^\times\cdot J^1_2)\\
Z\cdot K & \text{~if~} \gamma\notin \mathbf{\mathsf{N}}_{\GL_2(\cO_2)}(\mathfrak{O}_2^\times\cdot J^1_2).
\end{cases}$
\end{center}
Note that there is a bijection between $\mathfrak{O}_2^\times\backslash \GL_2(\mathfrak{o}_2)/\mathfrak{O}_2^\times\cdot J^1_2$ and $\F_{q^2}^\times\backslash \GL_2(\F_q)/\F_{q^2}^\times$ for which a set of representatives $\Gamma_{0}$ are given in Lemma \ref{Various double cosets1}.
Since the index of $\F_{q^2}^\times$ in 
$\mathbf{\mathsf{N}}_{\GL_2(\F_q)}(\F_{q^2}^\times)$ is 2, we get the index of $\mathfrak{O}_2^\times$ in  $\mathbf{\mathsf{N}}_{\GL_2(\cO_2)}(\mathfrak{O}_2^\times)$ is 2. 
Now, the lemma follows by observing that $\theta|_{Z \cdot K} = \tilde{\phi}_{x}|_{Z \cdot K}$.
\end{proof}

\begin{corollary} \label{phi_2X_1 representation in Pi}
Let $\tilde{\phi}_{2X_1} : I(\phi_{2X_1}) \rightarrow \C^{\times}$ be given by $\tilde{\phi}_{2X_1}=\theta\cdot\phi_{2X_1}$ and $\sigma=\Ind_{\mathfrak{O}_2^\times\cdot J^1_2}^{\GL_2(\mathfrak{o}_2)}(\theta\cdot\phi_{2X_1})$, then $\sigma$ is a subrepresentation of $\Pi$.
Moreover, if $X_1\in \F_{q^2}\backslash\F_q$, then $\sigma$ is an irreducible representation of $\GL_2(\cO_2)$ and $\sigma$ appears in $\Pi$ with multiplicity one.
\end{corollary}

\begin{proof}
First part follows easily from the transitivity of induction.
For $X_1 \in \F_{q^2} \backslash \F_{q}$, $B= 2X_{1}$ is non-split semisimple, we use Equation (\ref{Inner product of characters in Pi}) to compute $\left\langle \Theta_{\Pi},\Theta_{\sigma}\right\rangle_{\GL_2(\cO_2)}$.
Since $\tilde{\phi}_{2X_1}=\theta\cdot\phi_{2X_1}$and $\theta$ is a strongly primitive character of $\mathfrak{O}_2^\times$, we get 
$$
\left\langle\theta,\tilde{\phi}_{2X_1} \right\rangle_{\mathfrak{O}_2^\times}=1
\text{~and~} \left\langle\theta,\tilde{\phi}_{2X_1}^{\gamma_0^{-1}} \right\rangle_{\mathfrak{O}_2^\times}=\left\langle\theta^{\gamma_0},\tilde{\phi}_{2X_1} \right\rangle_{\mathfrak{O}_2^\times}= \left\langle \theta^{\gamma_{0}}, \theta \right\rangle_{\mathfrak{O}_2^\times} =0 .$$
It remains to prove that $\langle \theta,\tilde{\phi}_{2X_1}^{\gamma^{-1}}\rangle_{Z\cdot K}=0$ for $\gamma \notin \{ I, \gamma_{0} \}$, which can be easily verified using Equation (\ref{two characters restricted to K}) and we skip the details. 
\end{proof}

\begin{corollary}
Let $B \in M_{2}(\F_{q})$ be a non-split semisimple matrix with trace $\tr(2X_1)$ and $B\nsim 2X_1$. 
Let $\tilde{\phi}_{B}$ be an extension of $\phi_{B}$ to its inertia group $I(\phi_{B})$  such that $\tilde{\phi}_{B}|_{Z} = \omega_{\Pi}$ and $\sigma=\Ind_{I(\phi_B)}^{\GL_2(\mathfrak{o}_2)}\tilde{\phi}_B$.
Then, $\sigma$ appears in $\Pi$ with multiplicity one.
\end{corollary}
\begin{proof}
Since $B$ is non-split semisimple matrix, we can take that $B \in \F_{q^2}\backslash\F_q$.
We use Equation (\ref{Inner product of characters in Pi}) to prove 
$\langle \Theta_{\Pi},\Theta_{\sigma}\rangle_{\GL_2(\cO_2)}=1$.
We first claim that 
$\left\langle\theta,\tilde{\phi}_B \right\rangle_{\mathfrak{O}_2^\times}=0$, 
$\left\langle\theta,\tilde{\phi}_B ^{\gamma_0^{-1}}\right\rangle_{\mathfrak{O}_2^\times}=0$. 
Since $1+ \varpi \mathfrak{O}_2 \subset \mathfrak{O}_{2}^{\times}$, it is enough to prove that $\left\langle\theta,\phi_B \right\rangle_{1+\varpi \mathfrak{O}_2}=0$ and $\left\langle\theta,\phi_B ^{\gamma_0^{-1}}\right\rangle_{1+\varpi \mathfrak{O}_2}=0$.
We identify $1+\varpi \mathfrak{O}_2$ with $\F_{q^2}$ under the isomorphism $J^1_2\cong M_2(\F_q)$. 
Recall that $\tr : M_2(\F_q)\times M_2(\F_q)\rightarrow \F_q$ is a non-degenerate bilinear map and so is its restriction to $\F_{q^2}$. 
Therefore, $B_1,~B_2 \in \F_{q^2}$ satisfy the following
\begin{center}
$\phi_{B_1}|_{\F_{q^2}} = \phi_{B_2}|_{\F_{q^2}}$ if and only if $B_1=B_2$. 
\end{center}
Note that $\theta|_{1+\varpi \mathfrak{O}_2}=\phi_{2X_1}|_{1+\varpi \mathfrak{O}_2}$. 
Since $B \nsim 2X_1$, we get $\phi_B|_{1+\varpi \mathfrak{O}_2} \neq \phi_{2X_1}|_{1+\varpi \mathfrak{O}_2}$ and ${\phi_B^{\gamma_{0}^{-1}}|_{1+\varpi \mathfrak{O}_2}} \neq \phi_{2X_1}|_{1+\varpi \mathfrak{O}_2}$.
Hence the claim follows.

One can easily verify that there exists unique $\gamma \in \Gamma_{0} \backslash \{ I, \gamma_0 \}$ with $\left\langle\tilde{\phi}_x,\tilde{\phi}_B^{\gamma^{-1}}\right\rangle_{Z\cdot K}=1$, for which we skip the details.
\end{proof}

\begin{lemma}\label{split semisimple sigma in Pi}
Let $B = \begin{pmatrix}
    m_1 & 0\\
    0 & 4(a_0+aa_1)-m_1
\end{pmatrix}$ be a split semisimple matrix such that $m_1\neq 2(a_0+aa_1)$. 
Let $\tilde{\phi}_{B}$ be an extension of $\phi_{B}$ to its inertia group $I(\phi_{B})$  such that $\tilde{\phi}_{B}|_{Z} = \omega_{\Pi}$ and $\sigma=\Ind_{I(\phi_B)}^{\GL_2(\mathfrak{o}_2)}\tilde{\phi}_B$. 
Then, $\sigma$ appears in $\Pi$ with multiplicity one.
\end{lemma}
\begin{proof}
We use Equation (\ref{Equation for Pi,sigma}) by noting that for the given $B$, the inertia group $I(\phi_{B})$ is such that $I(\phi_{B})/J^{1}_{2}$ is the diagonal torus $T_{1}$ in $\GL_2(\F_{q})$.
Moreover, there is a bijection between $\mathfrak{O}_2^\times\backslash \GL_2(\cO_2)/I(\phi_B)$ and $\F_{q^2}^\times\backslash\GL_2(\F_q)/T_1$ for which a set of representatives $\Gamma_{1}$ is known from Lemma \ref{Various double cosets1} (c).
One can easily verify that that there exists a unique $\gamma \in \Gamma_{1}$ satisfying 
$\left\langle \theta,\tilde{\phi}_B^{\gamma}\right\rangle_{\gamma I(\phi_B)\gamma^{-1}\cap \mathfrak{O}_2^\times}=1$.
\end{proof}

\begin{lemma}
Let $X_1\in \F_{q^2}\backslash \F_q$ and $B=\begin{pmatrix}
    2(a_0+aa_1) & 1\\
    0 & 2(a_0+aa_1)
\end{pmatrix}$. Let $\tilde{\phi}_{B}$ be an extension of $\phi_{B}$ to its inertia group $I(\phi_{B})$ such that $\tilde{\phi}_{B}|_{Z} = \omega_{\Pi}$ and $\sigma = \Ind_{I(\phi_B)}^{\GL_2(\mathfrak{o}_2)}\tilde{\phi}_B$. 
Then, $\sigma$ appears in $\Pi$ with multiplicity one.
\end{lemma}
\begin{proof}
For the given $B$ the inertia group $I(\phi_{B})$ is such that $I(\phi_{B})/J^{1}_{2} = T_2$ as defined in Lemma \ref{Various double cosets1}. The rest of the proof is similar to that of Lemma \ref{split semisimple sigma in Pi}.
\end{proof}

We summarise the results from Section \ref{Sec : 7} in the following theorem.
\begin{theorem}\label{Complete description of Pi}
Let $\Pi = \Ind_{\mathfrak{O}_2^\times}^{\GL_2(\mathfrak{o}_2)}(\theta|_{\mathfrak{O}_2^\times})$. Recall that $\theta = \phi_{2X_1}$ when restricted to $\mathfrak{O}_{2}^{\times} \cap J^{1}_{2} \cong 1+ \varpi \mathfrak{O}_2$.
\begin{enumerate}[label = {(\alph*)}]
\item Let $X_1\in \F_q$, then $\Pi$ consists of the following representations of $\GL_2(\cO_2)$.
\begin{enumerate}[label = {(\roman*)}]
\item The representation $\Ind_{\mathfrak{O}_2^\times \cdot J^{1}_{2}}^{\GL_2(\cO_2)} \left( \tilde{\phi}_{2X_1} \right)$, which is not a regular representation and it is also not irreducible.
\item An irreducible regular representation $\sigma$ of $\GL_2(\cO_2)$ with multiplicity one such that $\omega_{\sigma}=\omega_{\Pi}$ and $\sigma|_{J^{1}_{2}}$ contains the character $\phi_B$ where $B$ is a semisimple matrix of trace $\tr(2X_1)$.
\end{enumerate}

\item Let $X_1\in \F_{q^2}\backslash\F_q$, then $\Pi$ consists of the following representations of $\GL_2(\cO_2)$.
\begin{enumerate}[label = {(\roman*)}]
\item  The irreducible regular representation $\Ind_{\mathfrak{O}_2^\times \cdot J^{1}_{2}}^{\GL_2(\cO_2)} \left( \tilde{\phi}_{2X_1} \right)$ with multiplicity one.
\item An irreducible regular representation $\sigma$ of $\GL_2(\cO_2)$ with multiplicity one such that $\omega_{\sigma}=\omega_{\Pi}$ and $\sigma|_{J^{1}_{2}}$ contains the character $\phi_B$ where $B$ is a regular matrix of trace $\tr(2X_1)$ and $B \nsim 2X_1$. 
\end{enumerate}
\end{enumerate}
\end{theorem}
Using Theorem \ref{Complete description of pi_N,psi} and Theorem \ref{Complete description of Pi}, we obtain our main result verifying the Conjecture \ref{DP conjecture} of Prasad for the case $n=2, l=2$.
\begin{theorem}
Let $\pi$ be an irreducible strongly cuspidal representation of $\GL_{4}(\cO_2)$ associated to a strongly primitive character $\theta : \mathcal{O}_2^\times\rightarrow \C^{\times}$. 
Then as a representation of $\GL_{2}(\cO_2)$ we have 
$$
\pi_{N, \psi} \cong \Ind_{\mathfrak{O}_2^\times}^{\GL_{2}(\cO_2)} (\theta|_{\mathfrak{O}_2^\times}).
$$
\end{theorem}
\section*{Acknowledgement}
The authors thank Dipendra Prasad for several helpful conversations and insightful comments. 
This work is a part of AP's Thesis at IIT Delhi and she acknowledges the support of PMRF fellowship (1400772). 
SPP acknowledges the support of the SERB MATRICS grant (MTR/2022/000782).

\bibliographystyle{abbrv}
\bibliography{refs}

\end{document}